\documentclass[onefignum,onetabnum]{amsart}

\usepackage{lipsum}
\usepackage{amsfonts}
\usepackage{graphicx}
\usepackage{epstopdf}
\usepackage{algorithmic}
\ifpdf
  \DeclareGraphicsExtensions{.eps,.pdf,.png,.jpg}
\else
  \DeclareGraphicsExtensions{.eps}
\fi

\title{Space-time evolution of Volterra disclinations}
\author[P.~Cesana]{Pierluigi Cesana}
\address[P.~Cesana]{Institute of Mathematics for Industry, Kyushu University, 744 Motooka, Fukuoka 819-0395, Japan}
\email{cesana@math.kyushu-u.ac.jp}
\author[A.~Grillo]{Alfio Grillo}
\author[M.~Morandotti]{Marco Morandotti}
\address[A.~Grillo, M.~Morandotti, A.~Pastore]{Politecnico di Torino, Corso Duca degli Abruzzi, 24, 10129 Torino, Italy}
\email{$\{$alfio.grillo,marco.morandotti,andrea.pastore$\}$@polito.it}
\author[A.~Pastore]{Andrea Pastore}

\usepackage{amsopn}

\makeatletter
\newcommand*{\addFileDependency}[1]{
  \typeout{(#1)}
  \@addtofilelist{#1}
  \IfFileExists{#1}{}{\typeout{No file #1.}}
}
\makeatother

\usepackage{amsmath}           
\usepackage{amsfonts}          
\usepackage{amssymb}           
\usepackage{amsthm}            
\usepackage{accents}           
\usepackage{array}             
\usepackage[T1]{fontenc}       
\usepackage{mathtools}
\usepackage{color}             
\usepackage{verbatim}          

\usepackage[english]{babel}
\usepackage[percent]{overpic}
\usepackage{accents}

\usepackage{xcolor}
\definecolor{orange}{rgb}{1.0, 0.49, 0.0}

\usepackage{hyperref}

\usepackage{lineno}            
\usepackage{appendix}
\usepackage{upgreek}
\usepackage{graphicx}
\usepackage{subcaption}
\usepackage{color,soul}

\usepackage{datetime}

\definecolor{green}{rgb}{0.0, 0.5, 0.0}

\newtheorem{theorem}{Theorem}[section]
\newtheorem{lemma}[theorem]{Lemma}
\theoremstyle{remark}
\newtheorem{remark}{Remark}
\theoremstyle{definition}
\newtheorem{definition}{Definition}[section]

\allowdisplaybreaks

\usepackage{xr}

\newcommand{\Bilap}{\Delta^{2}}
\newcommand{\Lap}{\Delta}

\newcommand{\defeq}{\coloneqq}

\newcommand{\xidot}{\dot{\xi}}

\newcommand{\Gcal}{\mathcal{G}}

\newcommand{\Hhat}{\hat{H}}

\begin{document}

\maketitle

\begin{abstract}
  The dynamics of a system of particles subject to a 4th order potential field modeling the space-time evolution of wedge disclinations is studied, focusing on finite systems of disclinations within a circular domain.
  Existence theorems for the trajectories of these disclinations are presented, considering both the 
  dynamics without predefined preferred directions of motion in an isotropic medium and the dynamics in which the disclinations move parallel to predefined directions, modeling a crystalline material.
  The analysis is illustrated with a number of numerical solutions to demonstrate various relevant configurations.
\end{abstract}


\noindent
\textbf{2020 AMS Classification:}
 70F40, 
 74B99, 
 49J10, 
 34A60. 

%




%


\section{Introduction} \label{sec_Introduction}
Wedge disclinations, defined by the corresponding Frank angle, are asymmetries observed at the lattice level in various materials and systems relevant to technological applications, such as elastic crystals \cite{Orowan1934,Polanyi1934,Taylor1934}, shape memory alloys \cite{I19,IHM13,ILTH17,IS20}, 
metal alloys \cite{ET67,TE68}, graphene \cite{Banhart11,Yang18}, and more. Alongside dislocations (translational defects), disclinations represent the classical lattice rotational   irregularities envisioned by Volterra \cite{V07}.

The investigation and modeling of disclinations, particularly regarding their formation, motion, and interaction with other defects, are crucial for understanding the structure-property relationships in the corresponding materials. For instance, in graphene wedge disclinations with opposite angles tend to attract and interact, forming disclination dipoles \cite{RRK2018}, 
Stone--Wales defects \cite{KVV2016}, 
and even more complex configurations commonly referred to as ``flowers'' \cite{LBSL2013}, 
which dramatically influence electrical/magnetic/mechanical properties and chemical reactivity. In metal alloys and elastic crystals, moving disclination dipoles create shear similar to that of deformation twins, in the context of accommodating plastic deformations \cite{Smirnov2022,TOKUZUMI2023118785},
thereby affecting mechanical behavior, such as stress concentration in the system.

While dislocation dynamics has 
been extensively studied by 
experimentalists (see \cite{AKHONDZADEH2023118851,Dorn1965,ROBACH20061679} and the references therein), modelists \cite{ACHARYA2001761,ACHARYA2010766,ZHANG2015145}
and theoreticians \cite{ADLGP2017,Blass2015a,Cermelli1999a,GvMPS2020,HM2017,MPS2017,PS2021,vMM2019,vMP2024} (see also \cite{Yavari_ARMA} for a geometric approach), the investigation of disclination dynamics has 
received less attention.
We refer to \cite{BCH15,CH20} for a probabilistic model of the space-time evolution of self-similar martensitic microstructures, which are systems known to be associated with lattice mismatches, particularly disclinations.
A general mathematical theory of moving disclinations in linear elasticity has been presented in 
\cite{KdW1977b,KdW1977a},
where equations of motion are expressed in terms of the total (i.e., elastic + plastic) displacement
(see also the 
references to earlier works contained therein, as well as 
\cite{RV1983}, 
and \cite{AF2015,FRESSENGEAS20113499,Taupin2013}
). For a nonlinear approach, see, e.g., \cite{Yavari_MMS}.
In \cite{K1979}, 
forces acting on defects interacting with external fields of mechanical stresses are calculated. It is found that wedge disclinations do interact with the surface of the elastic body and tend to get ejected from the domain boundary.
To the best of our knowledge a systematic analysis of these equations is not available as of yet.

Clarifying the dynamics of disclinations is essential for understanding microstructure evolution and grain boundary migration. In the context of metallurgy, it is crucial for developing better metal alloys in terms of strength, ductility, toughness, and more. Additionally, it allows for the control, inverse design, and optimization of graphene structures for specific applications \cite{AJE2021}.

The goal of this paper is to get insight from a simple conceptual model of 
the time-space evolution of a system of disclinations, and to study some of the properties of such a system, 
e.g., dynamic equilibrium and stability, through the solution of a few selected benchmarks.

We follow the mathematical modeling presented in \cite{Cesana2024a}, 
which is based on the mechanical modeling work of \cite{W68,SN88,V07}. 
We operate under the assumption of linearized kinematics in a planar strain regime within a homogeneous and isotropic continuum 
medium, in which 
wedge disclinations are described by concentrated singularities that are point sources of kinematic incompatibility. Precisely, each wedge disclination is represented by a Dirac delta function located at point $\xi_{i}$, with a multiplicative factor $s_i \in \mathbb{R} \setminus \{0\} $ denoting  
the Frank angle, which corresponds to the angular mismatch. With some abuse of notation, we will refer to the Frank angles as (signed) charges, in analogy with electrostatics.

In conditions of mechanical equilibrium, we describe the mechanical stress using the Airy potential function of the system \cite{Cesana2024a,Michell,V07}. Under these circumstances, and in analogy with the second-order case of the electrostatic potential, we find that the dynamics of a finite system of wedge disclinations is modeled by the space-time evolution of a finite collection of charges under the effect of the fourth-order Airy potential field generated by the disclinations themselves. In analogy with \cite{Cermelli1999a}, we assume that the dynamics is of dissipative type and 
driven by the negative gradient of the elastic energy of the mechanical system of disclinations \cite{Blass2015a,BlassMorandotti17}.


\textcolor{black}{Although our mechanical modeling approach is based on the 
assumptions of isotropic elasticity and linearized kinematics in the plane strain regime, the resulting time-space dynamics for the disclinations is highly non-linear. Consequently, the evolution and interactions of the disclinations are non-trivial, even in the case of simple systems. Our analysis can effectively capture fundamental dynamical mechanisms, such as the movement of a single disclination driven by forces towards the domain boundary, dipole and tripole interactions, and internal reorganization within large (finite) systems of disclinations. These mechanisms serve as benchmarks for describing more complex configurations of disclinations in realistic scenarios.}

The outline of the paper is as follows.
After introducing the main notation and recalling the mechanical modeling of disclinations in Section \ref{sec_ProblemDescription}, we formulate the mechanical energy for finite systems of disclinations via \emph{superposition} in Section \ref{sec_EnergiesSystemDisclinations}. Section \ref{sec_DisclinationDynamics} contains our main results on the dynamics in various configurations, including constrained dynamics; 
existence theorems, exact solutions for the case of a single disclination in a circular domain, and exact limits for disclination dipoles are illustrated. 
In Section \ref{sec_glideDirections}, we consider the case in which the dynamics follows preferred directions. Finally, Section \ref{sec_simulations} presents numerical simulations and describes some remarkable examples.

\section{Problem description}
\label{sec_ProblemDescription}
In this section, we briefly recall the main ingredients needed to establish the mechanical setting developed by Cesana et al.~\cite{Cesana2024a} to study the dynamics of disclinations in \emph{plane strain} through a variational theory of concentrated defects in elastic media undergoing infinitesimal deformations. 

Let us consider an open, bounded, and simply connected set $\Omega\subset \mathbb{R}^{2}$, and let $v\in H^{2}_{0}(\Omega)$ denote the Airy potential function describing the plane strain linear elastic material under study. Then, the elastic energy of the medium in terms of the Airy potential function reads \cite{Cesana2024a}
\begin{linenomath}
\begin{align}
\mathcal{W}(v;\Omega) \coloneqq \frac{1}{2}\frac{1+\nu}{E}\int_{\Omega}(|\nabla^{2} v|^{2}-\nu(\Delta v)^{2}) = \frac{1}{2}\frac{1-\nu^{2}}{E}\int_{\Omega} (\Lap v)^{2},
\label{eq_ElasticEnergyFunctional}
\end{align}
\end{linenomath}
where $E$ and $\nu$ are the material's Young modulus and Poisson ratio, respectively, $\nabla^{2}v$ is the Hessian of $v$, and $\Delta v$ is the Laplacian of $v$. Note that, granted the hypothesis of homogeneous material (i.e., with $\nu$ constant), the second equality holds true for functions $v$ in $H^{2}_{0}(\Omega)$, because, owing to the properties of such functions, it can be shown that $\int_{\Omega}|\nabla^{2}v|^{2}$ is equal to $\int_{\Omega}(\Delta v)^{2}$. We recall that the Airy potential permits to express the components of Cauchy stress in a way that satisfies automatically the mechanical equilibrium equation. 

Following 
\cite{Cesana2024a}, to account for punctual wedge disclinations in $\Omega$, we introduce the \emph{incompatibility measure} $\theta$, which describes the 
kinematic incompatibility associated with the \emph{rotational} symmetry breaking 
of the material at the point in which a wedge disclination appears. To define such disclinations properly, we use the concept of \emph{Frank angle} \cite{RV92}, i.e., the angle that defines the aperture of a wedge-type defect of the material lattice from its ideal, regular, arrangement. In this regard, we consider the set $\mathcal{S}\subset \mathbb{R}\setminus\{0\}$ of admissible, or possible, values taken by the Frank angles associated with each wedge disclinations, and we give the following: 

\begin{definition}[Disclination --- adapted from \cite{Cesana2024a}]
A disclination in $\Omega$ 
is a pair $d \equiv (s, \xi) \in \mathcal{S}\times \Omega$, where $s\in\mathcal{S}$ is the Frank angle associated with the disclination, and $\xi \in \Omega$ is the point in which the disclination is placed.
\end{definition}

If we assume the presence of $N\in\mathbb{N}$ disclinations in $\Omega$, denoted by $d_{1},\ldots, d_{N}$, so that $d_{k} = (s_{k},\xi_{k})$, with $k = 1,\ldots,N$
, the incompatibility measure is  \cite{Cesana2024a,RV92,SN88}
\begin{linenomath}
    \begin{align}
        \theta \equiv \sum_{k=1}^{N}\theta_{k} \coloneqq  \sum_{k=1}^{N}s_{k}\delta_{\xi_{k}}, &&\theta_{k} \coloneqq s_{k}\delta_{\xi_{k}},
        \label{eq_incompatibilityMeasure}
    \end{align}
\end{linenomath}
where, for each $k = 1,\ldots, N$, $\delta_{\xi_{k}}$ represents the Dirac distribution located in $\xi_{k}\in\Omega$ and  has physical dimensions of the reciprocal of a length squared.
Notice that in measuring the Frank angles we adopt the notation of \cite{Cesana2024a}.

Upon taking the duality product $\langle\theta, v\rangle$, which, by virtue of \eqref{eq_incompatibilityMeasure}, returns $\langle\theta, v\rangle = \sum_{k=1}^{N}s_{k}v(\xi_{k})$, we can define the energy functional \cite{Cesana2024a} 
    \begin{linenomath}
        \begin{align}
            \mathcal{I}^{\theta}(v;\Omega) \coloneqq \mathcal{W}(v;\Omega) + \langle\theta, v\rangle = \mathcal{W}(v;\Omega) + \sum_{k=1}^{N}s_{k}v(\xi_{k}).
        \end{align}
    \end{linenomath}    
This represents the total energy of the system and is composed by the elastic energy $\mathcal{W}(v;\Omega)$ and the energy associated with the incompatibilities $\langle\theta, v\rangle$.  

The partial differential equation governing the system is the Euler-Lagrange equation obtained by imposing the vanishing of the first variation of $\mathcal{I}^{\theta}(\cdot,\Omega)$ with respect to the Airy potential. Its solution $\bar{v}$ is sought for in $H^{2}_{0}(\Omega)$ by solving the biharmonic boundary value problem \cite{Cesana2024a}
\begin{linenomath}
    \begin{align}
    \frac{1-\nu^2}{E}\Bilap v = -\theta \quad \textrm{in $\Omega$}, &&
    v = 0 \quad\textrm{on $\partial\Omega$}, &&
   \partial_{\boldsymbol{n}}{v} = 0 \quad\textrm{on $\partial\Omega$},
   \label{eq_EquilibriumProblem}
    \end{align}
\end{linenomath}
and, due to the strict convexity of $\mathcal{I}^{\theta}(\,\cdot\,;\Omega)$ with respect to $v \in H^{2}_{0}(\Omega)$, is such that~\cite{Cesana2024a}
\begin{linenomath}
\begin{align}
\bar{v} = \mathrm{argmin}\{\mathcal{I}^{\theta}(v;\Omega)\,|\,v\in H^{2}_{0}(\Omega)\}.
\label{eq_MinimumProblem}
\end{align}
\end{linenomath}

Upon setting $W\coloneqq\mathcal{W}(\bar{v};\Omega)$ to indicate the \emph{actual} elastic energy of the system, i.e., the one determined by substituting  the Airy potential that solves \eqref{eq_EquilibriumProblem} into the elastic energy functional $\mathcal{W}(\,\cdot\,;\Omega)$, we have recourse to the following:
\begin{theorem}[Clapeyron's theorem in $H^{2}_{0}(\Omega)$]\label{thm_Clapeyron}
    Given $\bar{v}\in H^{2}_{0}(\Omega)$ the solution to  \eqref{eq_EquilibriumProblem}, it follows that the \emph{actual} elastic energy $W \defeq \mathcal{W}(\bar{v};\Omega)$ reads
    \begin{linenomath}
    \begin{align}
        W = -\tfrac{1}{2}\langle\theta,\bar{v}\rangle.
        \label{eq_ElasticEnergySolution}
    \end{align}
\end{linenomath}
\end{theorem}
\begin{proof} Upon testing  \eqref{eq_EquilibriumProblem} against the solution $\bar{v}$, we obtain 
\begin{linenomath}
        \begin{align}
            -\langle\theta,\bar{v}\rangle = \frac{1-\nu^2}{E}\int_{\Omega}(\Bilap \bar{v})\bar{v} = \frac{1-\nu^2}{E}\int_{\Omega}(\Lap \bar{v})^2 = 2W,
            \label{eq_centralDisclination_5}
        \end{align}
\end{linenomath}
where the second equality follows from the boundary conditions in \eqref{eq_EquilibriumProblem}.
\end{proof}

Under the assumption that the Frank angles of all the considered disclinations are constant, Clapeyron's Theorem in the form of \eqref{eq_ElasticEnergySolution} permits to express the \emph{actual} elastic energy $W$ in terms of the positions of the $N$ disclinations as 
\begin{linenomath}
    \begin{align}
        W \equiv \hat{W}(\xi_{1},\ldots, \xi_{N}) = -\frac{1}{2}\sum_{k=1}^{N}s_{k}\bar{v}(\xi_{k}).
        \label{sec_energyDisclinations}
    \end{align}
\end{linenomath}

\section{Elastic energy for fixed-in-space system of disclinations}
\label{sec_EnergiesSystemDisclinations}

We focus now on the computation of the Airy potential function and of the \emph{actual} elastic energy for a system of $N$ disclinations in  $\Omega$, hereafter assumed to be the open ball of radius $R>0$ centered in the origin, i.e., $\Omega\equiv B_{R}(0)$. To this end, we consider $N$ disclinations $d_{1},\ldots, d_{N}$ assumed to be fixed in space, as if we were looking at a \emph{photograph} of the system, and \eqref{sec_energyDisclinations} represents the elastic energy in the \emph{static} scenario. 

%
\subsection{The single disclination scenario}\label{sec_OneDisclination}
As a point of departure, we consider the problem of a single disclination ($N=1$) placed in $\xi\in\Omega$, and fixed in this point. In this case, the incompatibility measure $\theta$ in \eqref{eq_EquilibriumProblem} reduces to the single Dirac delta centered in $\xi$, i.e., $\theta \equiv s\delta_{\xi}$. For $\xi=0$, the corresponding Airy potential was computed in \cite{Cesana2024a}. However,
the generalization to a disclination fixed in space and located at $\xi\neq 0$ can be achieved by determining the Green function associated with the distribution $\theta \equiv s\delta_{\xi}$ (see e.g. the \emph{``clamped disk''} problem \cite[Section 4]{Nakai1978}). In this case, the  solution $\bar{v}\in H^{2}_{0}(\Omega)$ to  
\eqref{eq_EquilibriumProblem} reads
\begin{linenomath}
        \begin{align}
            \bar{v}(x) \coloneqq \begin{cases}
                \bar{u}(x), &x\in\Omega\setminus\{\xi\},\\
                \bar{w}, &x = \xi,
            \end{cases}
            \label{eq_eccentricDisclination_2}
        \end{align}
\end{linenomath}
where $\bar{u}(x)$ is defined, for any $x\in\Omega\setminus\{\xi\}$, as 
\begin{linenomath}
        \begin{align}
            \bar{u}(x) \defeq 
            &-\frac{E}{1-\nu^2}\frac{sR^{2}}{16\pi}\frac{|x-\xi|^2}{R^2}\log\frac{|x-\xi|^2}{R^2}-\frac{E}{1-\nu^2}\frac{sR^{2}}{16\pi}\left(1-\frac{|x|^2}{R^2}\right)\left(1-\frac{|\xi|^2}{R^2}\right)\nonumber\\
            &+ \frac{E}{1-\nu^2}\frac{sR^{2}}{16\pi}\frac{|x-\xi|^2}{R^2}\log\frac{R^4-2R^2x\cdot\xi + |x|^2|\xi|^2}{R^4}, 
            \label{eq_eccentricDisclination_2bis}
        \end{align}
\end{linenomath}
and $\bar{w}$ is defined as the extension by continuity of $\bar{u}$ in $x = \xi$, i.e., 
\begin{linenomath}
    \begin{align}
        \bar{w} \coloneqq -\dfrac{E}{1-\nu^2}\dfrac{sR^{2}}{16\pi}\bigg(1-\dfrac{|\xi|^2}{R^2}\bigg)^2.
        \label{eq_eccentricDisclination_2bisbis}
    \end{align}
\end{linenomath}
The first term on the right-hand side of \eqref{eq_eccentricDisclination_2bis} is the Green function of the biharmonic problem in $\mathbb{R}^{2}$ with the singularity being situated in $x = \xi$; the second and the third terms are due to the boundedness of $\Omega$. In particular, the third term stems from the disclination not being placed in the origin, and features the quantity 
\begin{linenomath}
    \begin{align}
        \frac{R^{4}-2R^{2}x\cdot\xi+|x|^{2}|\xi|^{2}}{R^{4}} = \left(1-\frac{|x|^{2}}{R^2}\right)\left(1-\frac{|\xi|^{2}}{R^2}\right)+ \frac{|x-\xi|^2}{R^2},
        \label{eq_ImageChargeTerm}
    \end{align}
\end{linenomath}
which appears when applying the method of \textit{sources}, or of \textit{image charges}, to look for Green functions in bounded domains (see e.g. \cite[Chapter IV]{Tikhonov2013}).

By employing  Theorem \ref{thm_Clapeyron}, the \emph{actual} elastic energy reads
\begin{linenomath}
        \begin{align}
            W = -\tfrac{1}{2} s\bar{v}(\xi) = -\tfrac{1}{2} s\bar{w} = \frac{E}{1-\nu^2}\frac{s^2R^{2}}{32\pi}\left(1-\frac{|\xi|^{2}}{R^2}\right)^2,
            \label{eq_eccentriclDisclination_3}
        \end{align}
\end{linenomath}
and we can see that: if $\xi = 0$, the resulting elastic energy $W$ is the same as the one in computed in \cite{Cesana2024a}; the energy $W\equiv \hat{W}(\xi)$ depends on $\xi$ solely through the distance between the origin and $\xi$, i.e., $W\equiv \check{W}(|\xi|)$, thereby making $W$ \textit{invariant} with respect to \emph{rotations} of the disclination around the origin; moreover, $\check{W}$ has a maximum if $|\xi| = 0$, i.e., in the origin, and $\check{W}$ goes to zero \textit{quadratically} as $\xi$ approaches the boundary, i.e., $\check{W}(|\xi|) = O((R-|\xi|)^{2})$ for $|\xi| \to R$.

\subsection{The $N$ disclinations scenario}
\label{sec_systemDisclinations}
When $N$ disclinations are present in $\Omega$, then the Airy potential characterizing the system is given by the solution $\bar{v}\in H^{2}_{0}(\Omega)$ of the biharmonic problem in \eqref{eq_EquilibriumProblem} with the incompatibility measure specified in \eqref{eq_incompatibilityMeasure}. By the linearity of  \eqref{eq_EquilibriumProblem}, and by the \textit{homogeneity} of the associated Dirichlet and Neumann boundary conditions, it can be proved that $\bar{v}$ is given according to  \emph{superposition principle} as the summation $\bar{v} = \bar{v}_{1} + \cdots + \bar{v}_{N}$, where the $h$th function $\bar{v}_{h}$, with $h=1,\ldots,N$, is the solution of the $h$th subproblem 
\begin{linenomath}
    \begin{align}
    \frac{1-\nu^2}{E}\Bilap v = -\theta_{h} \quad \textrm{in $\Omega$}, &&
    v = 0 \quad\textrm{on $\partial\Omega$}, &&
   \partial_{\boldsymbol{n}}{v} = 0\quad \textrm{on $\partial\Omega$}.
   \label{eq_EquilibriumProblem_k_SubProblem}
    \end{align}
\end{linenomath}
Since each $\bar{v}_{h}$ is of the type reported in \eqref{eq_eccentricDisclination_2}, we conclude that, upon setting $C\coloneqq ER^2/(16\pi(1-\nu^{2}))$, the solution $\bar{v}$ is given by
\begin{linenomath}
        \begin{align}
            \bar{v}(x) = \sum_{h = 1}^{N}\bar{v}_{h}(x), &&\bar{v}_{h}(x)\coloneqq\begin{cases}
                \bar{u}_{h}(x), &x\in\Omega\setminus\{\xi_{h}\},\\
                \bar{w}_{h}, &x = \xi_{h},
            \end{cases} 
            \label{eq_ManyDisclinations_3}
        \end{align}
\end{linenomath}
where, for each $h=1,\ldots, N$, $\bar{u}_{h}(x)$ is defined for any $x\in\Omega\setminus\{\xi_{h}\}$ as  
\begin{linenomath}
        \begin{align}
            \bar{u}_{h}(x) =  
            &-Cs_{h}\frac{|x-\xi_{h}|^2}{R^2}\log\frac{|x-\xi_{h}|^2}{R^2}-Cs_{h}\left(1-\frac{|x|^2}{R^2}\right)\left(1-\frac{|\xi_{h}|^2}{R^2}\right)\nonumber\\
            &+Cs_{h}\frac{|x-\xi_{h}|^2}{R^2}\log\frac{R^{4}-2R^{2}x\cdot\xi_{h}+|x|^{2}|\xi_{h}|^{2}}{R^{4}},
            \label{eq_ManyDisclinations_3bis}
        \end{align}
\end{linenomath}
%
and $\bar{w}_{h}$ is defined by continuity as 
\begin{linenomath}
    \begin{align}
        \bar{w}_{h} \coloneqq -Cs_{h}\bigg(1-\dfrac{|\xi_{h}|^2}{R^2}\bigg)^2.
        \label{eq_ManyDisclinations_3bisbis}
    \end{align}
\end{linenomath}

By virtue of \eqref{eq_ManyDisclinations_3}, we can state the following:
\begin{lemma}[Reciprocity in $H^{2}_{0}(\Omega)$]\label{thm_Betti}
    Given $\theta = \theta_{1}+\cdots+\theta_{N}$ and the solution $\bar{v} = \bar{v}_{1}+\cdots+\bar{v}_{N}$ to \eqref{eq_EquilibriumProblem} reported in \eqref{eq_ManyDisclinations_3}, it follows that 
    \begin{linenomath}
        \begin{align}
            \langle\theta_{k},\bar{v}_{h}\rangle \equiv s_{k}\bar{v}_{h}(\xi_{k})=  s_{h}\bar{v}_{k}(\xi_{h}) \equiv \langle\theta_{h},\bar{v}_{k}\rangle, &&h,k=1,\ldots,N.
            \label{eq_Reciprocity}
        \end{align}
    \end{linenomath}
\end{lemma}
\begin{proof} It follows from direct computations.
\end{proof}
Note that the result in \eqref{eq_Reciprocity} leads to Betti's theorem of linear elasticity.

\begin{theorem}[Betti's theorem in $H^{2}_{0}(\Omega)$]
Given $\bar{v} = \bar{v}_{1}+\cdots+\bar{v}_{n}$, the \emph{actual} elastic energy $W$ of the system can be expressed as the following sum  
\begin{linenomath}
    \begin{align}
        W = \sum_{k=1}^{N}W_{k} + \sum_{k=1}^{N}\sum_{h=k+1}^{N}W_{kh},
        \label{eq_Betti}
    \end{align}
\end{linenomath}
where we have defined the energetic contributions
\begin{linenomath}
    \begin{subequations}\label{eq_3.12}
        \begin{align}
            W_{k} \coloneqq&\; \frac{1}{2}Cs_{k}^2\Big(1-\frac{|\xi_{k}|^{2}}{R^2}\Big)^2,
            \label{eq_ManyDisclinations_5}\\
    W_{kh} \coloneqq&\; Cs_{h}s_{k}\frac{|\xi_{h}-\xi_{k}|^2}{R^2}\log\frac{|\xi_{h}-\xi_{k}|^2}{R^2} +Cs_{h}s_{k}\Big(1-\frac{|\xi_{h}|^{2}}{R^2}\Big)\Big(1-\frac{|\xi_{k}|^{2}}{R^2}\Big)\nonumber\\
    &-Cs_{h}s_{k}\frac{|\xi_{h}-\xi_{k}|^2}{R^2}\log \frac{R^{4}-2R^{2}\xi_{k}\cdot\xi_{h}+|\xi_{k}|^{2}|\xi_{h}|^{2}}{R^{4}},
\label{eq_ManyDisclinations_6}
    \end{align}
   \end{subequations}
\end{linenomath}
and the argument of the logarithm in the last summand on \eqref{eq_ManyDisclinations_6} admits the same decomposition as the one shown in \eqref{eq_ImageChargeTerm}, with $x\equiv \xi_{k}$ and $\xi \equiv \xi_{h}$.
\end{theorem}
\begin{proof} By substituting the decomposition $\bar{v} = \sum_{h=1}^{N}\bar{v}_{h}$ into \eqref{sec_energyDisclinations} and employing Theorem \ref{thm_Betti}, the \emph{actual} elastic energy in \eqref{sec_energyDisclinations} becomes
\begin{linenomath}
        \begin{align}
            W = -\frac{1}{2}\sum_{k=1}^{N} s_{k}\bar{v}(\xi_{k}) = -\frac{1}{2}\sum_{k=1}^{N}\sum_{h=1}^{N} s_{k}\bar{v}_{h}(\xi_{k})&= -\frac{1}{2}\sum_{k=1}^{N} s_{k}\bar{v}_{k}(\xi_{k}) - \sum_{k=1}^{N}\sum_{h=k+1}^{N}s_{k}\bar{v}_{h}(\xi_{k})\nonumber\\
            &= -\frac{1}{2}\sum_{k=1}^{N} s_{k}\bar{w}_{k} - \sum_{k=1}^{N}\sum_{h=k+1}^{N}s_{k}\bar{u}_{h}(\xi_{k}),
            \label{eq_ManyDisclinations_4}
        \end{align}
\end{linenomath}
which, by virtue of \eqref{eq_ManyDisclinations_5} and \eqref{eq_ManyDisclinations_6}, proves the thesis. 
\end{proof}

According to Betti's theorem, which descends from the \emph{superposition principle}, the \emph{actual} elastic energy can be expressed as the sum of $N$ energies $W_{k}$, with $k=1,\ldots, N$, and $N(N-1)/2$ energies $W_{kh}$, with $h,k=1,\ldots, N$ and $k<h$. Moreover, Betti's theorem provides the following meaning for the energies in \eqref{eq_ManyDisclinations_5} and \eqref{eq_ManyDisclinations_6}: $W_{k}$ is the elastic energy associated with the $k$th subproblem \eqref{eq_EquilibriumProblem_k_SubProblem}, i.e., the one in which only the $k$th disclination $d_{k}$ is present in $\Omega$; $W_{kh}$ is twice the elastic energy associated with a ``fictitious'' system in which the Airy potential is $\bar{v}_{h}$ and the incompatibility measure is $\theta_{k}$, or vice versa. 
\subsection{Towards dynamics: \emph{unfixing} the disclinations}
\label{sec_unfixing}
    Until now we have studied the disclinations in $\Omega$ in the static case, i.e., when they were fixed in space. As done in \cite{Blass2015a}, the energies featuring in \eqref{eq_Betti} can be identified with the values taken by appropriate functions $\hat{W}$, $\hat{W}_{k}$, and $\hat{W}_{kh}$ of the placements of the disclinations $\xi_{1},\ldots,\xi_{N}$. Hence, by setting $\boldsymbol{\xi} \coloneqq (\xi_{1},\ldots,\xi_{N})\in \Omega^{N}$, we write
\begin{linenomath}
    \begin{align}
        W \equiv \hat{W}(\boldsymbol{\xi})= \sum_{k=1}^{N}\hat{W}_{k}(\boldsymbol{\xi}) + \sum_{k=1}^{N}\sum_{h=k+1}^{N}\hat{W}_{kh}(\boldsymbol{\xi} ),
        \label{eq_BettiFunctions}
    \end{align}
\end{linenomath}
and we notice that $\hat{W}$ is well defined in the connected open set $\mathcal{Y}\coloneqq \Omega^{N}\setminus\Pi$, where $\Pi$ is given by (see \cite[Equation~(2.13)]{Blass2015a}) %
\begin{linenomath}
    \begin{align}
        \Pi \coloneqq \bigcup_{k=1}^{N}\bigcup_{h=k+1}^{N}\Pi_{kh}, &&\Pi_{kh}\coloneqq\big\{\boldsymbol{\xi}\in\Omega^{N}\,\big|\,\xi_{k} = \xi_{h}, k\neq h\big\}.
        \label{eq_collision_set}
    \end{align}
\end{linenomath}
In the jargon of \cite{Blass2015a}, when two different disclinations, say $d_{k}$ and $d_{h}$, with $k\neq h$, share the same position $\xi_{k} = \xi_{h}$, one speaks of a \emph{``collision''}. 
However, we use here a slightly different terminology. Indeed, studying the dynamics of the disclinations (see Section \ref{sec_DisclinationDynamics}), we notice that two disclinations placed at different positions at some initial time 
tend not to intersect in finite time. On the other hand, it is possible to prepare a \emph{Gedankenexperiment} in which $d_{k}$ and $d_{h}$ have coincident positions, i.e., $\xi_{k}=\xi_{h}$, but, in general, different Frank angles $s_{k}\neq s_{h}$. In such a situation, we say that the two  disclinations \emph{superpose}, and we rename the set $\Pi$ as the \emph{superposition set} \cite{Qi2009}.

It is also worth noticing that $\hat{W}$ can be extended by continuity in all the points of $\Pi$, and, if $\boldsymbol{\xi}\in\Pi_{kh}$ for some $k,h = 1,\ldots, N$, with $k\neq h$, then the resulting energy  would feature a contribution, comprehensive of the two disclinations $d_{k}$ and $d_{h}$, that is equivalent to having one \emph{effective} disclination in $\xi_{k}=\xi_{h}$, with Frank angle $s_{k}+s_{h}$. In particular, if there are two disclinations only, $d_{1}$ and $d_{2}$, such that $s_{1} = -s_{2}$, then the two disclinations cancel each other out (i.e., they \emph{annihilate} each other, in the terminology of \cite{Qi2009}), thereby restoring the compatibility of the deformation at the point in which they superpose. However, this annihilation does not persist, in general, if other interactions are perceived. Indeed, even though one prepares an ``experiment'' (\emph{in silico}, for this work) in which two disclinations with opposite Frank angles are superposed at the initial time of their dynamics, this compatible configuration of the system breaks dynamically as soon as another disclination is introduced. If, on the one hand, this behavior is intuitive, because one switches from a two-disclination to a three-disclination system, on the other hand, the \emph{splitting} phenomenon of the two disclinations is important because it shows that having two disclinations with opposite Frank angles, and interacting with a third disclination, \emph{is not equivalent} to having one disclination only (the third one). This issue is addressed again in Remark \ref{rem_LimitCase_2}.

Equations \eqref{eq_ManyDisclinations_5} and \eqref{eq_ManyDisclinations_6} 
yield 
that the function $\hat{W}$ is  differentiable in $\mathcal{Y}$ with respect to each one of its entries 
$\xi_{1}, \ldots, \xi_{N}$. Hence, the elastic energy function $\hat{W}$ is continuous and differentiable in $\mathcal{Y}$ with respect to all of its arguments \emph{in the classical sense}, so that the gradients of $\hat{W}$ with respect to each disclination position $\xi_{1},\ldots, \xi_{N}$ are well defined in $\mathcal{Y}$. 

Note that, in principle, the Frank angle associated with each disclination may have dynamics of its own, thereby requiring $\hat{W}$, $\hat{W}_{k}$, and $\hat{W}_{kh}$, with $k,h=1,\ldots,N$, to be functions of $d_{1},\ldots, d_{N}$ altogether. However, in the rest of this work we assume that the Frank angles of the disclinations do not vary.

Before going further, we emphasize that the dynamics of the disclinations has been the subject of many investigations, involving the concepts of topological charges, phase transitions, and spontaneous symmetry breaking. In this respect, and within an approach to the disclination dynamics based on Statistical Mechanics, mention must be made of the celebrated KTHNY-theory \cite{Kosterlitz1973, Qi2009}. In the sequel, however, rather than following a statistical mechanical methodology, we study the dynamics of a system of $N$ disclinations viewed as \emph{charged point particles}.

\section{Disclination dynamics}
\label{sec_DisclinationDynamics}
 As put forward in Section \ref{sec_Introduction}, 
 we \emph{assume} that the dynamics of the $k$th disclination in $\Omega$, for $k = 1,\ldots, N$, is dissipative and driven by the negative of the gradient of $\hat{W}$ with respect to $\xi_{k}$, thereby yielding a force $\hat{f}_{k} \coloneqq -\nabla_{\xi_{k}}\hat{W}$, formally similar to the Peach-Koehler forces acting on dislocations~\cite{Maugin1993}. 

Thus, upon introducing the initial conditions $\xi_{k}(0) = \xi_{k0}$, with $k=1,\ldots,N$, the dynamics of the $N$ disclinations considered in $\Omega$ is governed by  Cauchy's problem \cite{Blass2015a}
\begin{linenomath}
    \begin{equation}
        \begin{cases}
            \xidot_{k}(t) = -\lambda_{k}\nabla_{\xi_{k}}\hat{W}(\xi_{1}(t), \ldots, \xi_{N}(t)), &
            k = 1,\ldots, N,\\
            \xi_{k}(0) = \xi_{k0}, &k = 1,\ldots, N,
        \end{cases}
    \label{eq_CauchyProblem}
   \end{equation}
\end{linenomath}
where $\xidot_{k}$ is the derivative of $\xi_{k}$ with respect to time, and $\lambda_{k}>0$ is a positive scalar quantity having the meaning of the inverse of a generalized viscosity coefficient. For the sake of simplicity, we assume all the coefficients $\lambda_{k}$, with $k=1,\ldots,N$, to be equal to some $\lambda>0$. Note that models of this type are already used in \cite{Cermelli1999a,Blass2015a} for dislocations.

While some analytical issues of \eqref{eq_CauchyProblem}, such as existence and uniqueness of the solution, are addressed in Section \ref{sec_nondim}, we mainly 
focus now on the modeling aspects of \eqref{eq_CauchyProblem} through the description of two \emph{archetypal} scenarios for a system of disclinations in $\Omega$: these are the  ``one-disclination problem'' and the ``two-disclination problem'', handled in Section \ref{sec_DynamicsOneDisclination} and in Section \ref{sec_TwoDisclinations}, respectively. Note that it is sufficient to consider just these two situations, since the case in which more than two disclinations are present can be reconducted to the two aforementioned ones. In this regard, some benchmarks with $N\geq2$ are presented in Section \ref{sec_simulations}.

\subsection{Non-dimensionalization}\label{sec_nondim}
We rewrite the system in \eqref{eq_CauchyProblem} in non-dimensional form so as to study the dynamics of the disclinations in terms of the rescaled variables $X \coloneqq x/R$ and $T \coloneqq (\lambda C/R^{2})t$. Hence, upon defining the non-dimensional position of the $k$th disclination as $\Xi_{k}\coloneqq \xi_{k}/R$, and scaling the energy $W$ accordingly as $W=CH$, we can recast \eqref{eq_CauchyProblem} as
\begin{linenomath}
    \begin{equation}
        \begin{cases}
            \dot{\Xi}_{k}(T) = -\nabla_{\Xi_{k}}\hat{H}(\Xi_{1}(T), \ldots, \Xi_{N}(T)), &
            k = 1,\ldots, N,\\
            \Xi_{k}(0) = \Xi_{k0}, &k = 1,\ldots, N,
        \end{cases}
    \label{eq_CauchyProblem_Adimensionalized}
   \end{equation}
\end{linenomath}
where we have now set $\dot{\Xi}_{k}\equiv \mathrm{d}\Xi_{k}/\mathrm{d}T$ (from now on, the superimposed dot denotes differentiation with respect to $T$ rather than $t$), and $H = \hat{H}(\Xi_{1},\ldots, \Xi_{N})$ is the \emph{actual} elastic energy of the non-dimensional system, which reads
\begin{linenomath}
    \begin{align}
        \hat{H}(\Xi_{1}, \ldots, \Xi_{N}) =& \,\frac{1}{2}\sum_{k=1}^{N}s_{k}^{2}(1-|\Xi_{k}|^{2})^2 + \sum_{k=1}^{N}\sum_{h=k+1}^{N}s_{k}s_{h}(1-|\Xi_{h}|^{2})(1-|\Xi_{k}|^{2})\nonumber\\
        &+\sum_{k=1}^{N}\sum_{h=k+1}^{N}s_{k}s_{h}|\Xi_{h}-\Xi_{k}|^{2}\log \Phi_{kh}.
        \label{eq_Energy_Nondimensional}
    \end{align}
\end{linenomath}
In \eqref{eq_Energy_Nondimensional}, we have expressed the non-dimensional elastic energy in terms of the quantities $\Phi_{kh}$, with $h,k=1,\ldots, N$, which are ratios depending on the positions $\Xi_{k}$ and $\Xi_{h}$, and are defined as
\begin{linenomath}
    \begin{align}
    \Phi_{kh} \equiv \hat{\Phi}(\Xi_{k}, \Xi_{h}) \coloneqq 
    & \frac{|\Xi_{k}-\Xi_{h}|^{2}}{|\Xi_{k}-\Xi_{h}|^{2} + (1-|\Xi_{k}|^{2})(1-|\Xi_{h}|^{2})}.
    \label{eq_Ratio}
    \end{align}
\end{linenomath}
Note that the symmetry property $\Phi_{kh} \equiv \hat{\Phi}(\Xi_{k}, \Xi_{h})=\hat{\Phi}(\Xi_{h}, \Xi_{k}) \equiv \Phi_{hk}$ holds true. Moreover, $\Phi_{kh}$ vanishes identically when $\Xi_{k}=\Xi_{h}$, and tends to unity from below in the limits $|\Xi_{k}|\rightarrow 1^{-}$ and/or $|\Xi_{h}|\rightarrow 1^{-}$, with $\Xi_{k}\neq \Xi_{h}$. Yet, we have to exclude the value $\Phi_{kh}=0$ to appropriately define the logarithm in \eqref{eq_Energy_Nondimensional}, so that we have $0 < \Phi_{kh} < 1$ . A last remark concerns the fact that the elastic energy in \eqref{eq_Energy_Nondimensional} can be extended 
by continuity to the case in which one, or more, disclinations are on the boundary of $B_{1}(0)$. In particular, if the positions $\Xi_{1},\ldots, \Xi_{N}$ are all on the boundary, then it holds that $\hat{H}(\Xi_{1},\ldots, \Xi_{N}) = 0$.

We emphasize that, in the present framework, \eqref{eq_CauchyProblem_Adimensionalized} can be obtained by employing the \emph{Extended Hamilton Principle} \cite{Lanczos1970a}. In other words, by starting from the Lagrangian function $L\equiv \hat{L}(\Xi_{1},\ldots,\Xi_{N})\coloneqq-\hat{H}(\Xi_{1},\ldots,\Xi_{N})$, and recognizing that, within the non-dimensional setting developed above, each velocity $\dot{\Xi}_{k}$ is, in fact, a normalized dissipative non-potential force linear in the corresponding velocity, one can rewrite the equation of motion for the $k$th disclination as $\dot{\Xi}_{k} = \nabla_{\Xi_{k}}\hat{L}(\Xi_{1},\ldots, \Xi_{N})$.

Due to the non-dimensionalization procedure, the $k$th disclination is indicated with $D_{k} = (s_{k}, \Xi_{k})$, and the set in which the disclinations are positioned is denoted by $B_{1}(0)$, i.e., the open ball of unit radius, centered in the origin. Finally, $\Hhat$ is well defined in $\mathcal{Y}\coloneqq[B_{1}(0)]^{N}\setminus \Pi$, where, with a slight abuse of notation, $\mathcal{Y}$ and $\Pi$ refer now to the non-dimensional setting. In particular, $\Pi$ denotes the \emph{non-dimensional superposition set} (compare with \cite[formula (2.13)]{Blass2015a}) 
\begin{linenomath}
    \begin{align}
        \Pi \coloneqq \bigcup_{k=1}^{N}\bigcup_{h=k+1}^{N}\Pi_{kh}, &&\Pi_{kh}\coloneqq\big\{(\Xi_{1},\ldots, \Xi_{N})\in[B_{1}(0)]^{N}\,\big|\,\Xi_{k} = \Xi_{h}, k\neq h\big\}.
        \label{eq_collision_set}
    \end{align}
\end{linenomath}

To close this section, we remark that, upon employing the expression of $\hat{H}$ in \eqref{eq_Energy_Nondimensional}, it is possible to introduce in \eqref{eq_CauchyProblem_Adimensionalized} the non-dimensional forces $\hat{F}_{k} \coloneqq -\nabla_{\Xi_{k}}\hat{H}$, with $k=1,\ldots, N$, which read 
\begin{linenomath}
        \begin{align}
            \hat{F}_{k}(\Xi_{1},\ldots, \Xi_{N}) = \hat{F}_{k}^{(1)}(\Xi_{k}) + \sum_{h=1,h\neq k}^{N}\hat{F}_{kh}^{(2)}(\Xi_{k},\Xi_{h}) + \sum_{h=1,h\neq k}^{N}\hat{F}_{kh}^{(3)}(\Xi_{k},\Xi_{h}),
            \label{eq_ExistenceTheorem_2}
        \end{align}
\end{linenomath}
where, given $k,h = 1,\ldots, N$ with $h\neq k$,  we have introduced the notation 
\begin{linenomath}
    \begin{subequations}\label{eq_4.7}
        \begin{align}
        &F_{k}^{(1)} \equiv \hat{F}_{k}^{(1)}(\Xi_{k}) = 2s_{k}^{2}(1-|\Xi_{k}|^{2})\Xi_{k},
        \label{eq_ExistenceTheorem_2_bis_a}\\
        &F_{kh}^{(2)} \equiv\hat{F}_{kh}^{(2)}(\Xi_{k},\Xi_{h}) = 2s_{k}s_{h}(1-\Phi_{kh})(1-|\Xi_{h}|^{2})\Xi_{k},
        \label{eq_ExistenceTheorem_2_bis_b}\\
        &F_{kh}^{(3)} \equiv \hat{F}_{kh}^{(3)}(\Xi_{k},\Xi_{h}) = 2s_{k}s_{h}(1-\Phi_{kh} +\log \Phi_{kh})(\Xi_{h}-\Xi_{k}).
        \label{eq_ExistenceTheorem_2_bis_c}
        \end{align}
    \end{subequations}
\end{linenomath}
Each force $\hat{F}_{k}$ comprises one contribution due solely to the interaction of the $k$th disclination with the boundary $\partial B_{1}(0)$  (see \eqref{eq_ExistenceTheorem_2_bis_a}), $N-1$ contributions that account for the effect of $\partial B_{1}(0)$ on $\Xi_{k}$, modulated by the remaining $N-1$ disclinations (see \eqref{eq_ExistenceTheorem_2_bis_b}), and $N-1$ terms stemming from the mutual interactions with the other disclinations (see \eqref{eq_ExistenceTheorem_2_bis_c}). For future use, we also define
\begin{linenomath}
    \begin{align}
       F_{k}^{(2)} \equiv \hat{F}_{k}^{(2)}(\Xi_{1},\ldots, \Xi_{N}) \coloneqq \sum_{h=1,h\neq k}^{N}\hat{F}_{kh}^{(2)}(\Xi_{k},\Xi_{h}),
       \label{eq_tot_force_interaction_a}\\
       F_{k}^{(3)} \equiv \hat{F}_{k}^{(3)}(\Xi_{1},\ldots, \Xi_{N}) \coloneqq \sum_{h=1,h\neq k}^{N}\hat{F}_{kh}^{(3)}(\Xi_{k},\Xi_{h}).
       \label{eq_tot_force_interaction_a}
    \end{align}
\end{linenomath}

Since each $\hat{F}_{k}$ does not depend explicitly on time and is Lipschitz-continuous, uniformly with respect to time, in every compact subset of $\mathcal{Y}$ of the type $\mathcal{R}$, given by
\begin{linenomath}
    \begin{align}
        \mathcal{R}\coloneqq 
        \bigtimes_{k=1}^{N}\big([\Xi_{k0}^{1}-b_{k}^{1}, \Xi_{k0}^{1}+b_{k}^{1}]\times[\Xi_{k0}^{2}-b_{k}^{2}, \Xi_{k0}^{2}+b_{k}^{2}]\big),
        \label{eq_compactSet}
    \end{align}
\end{linenomath}
and having empty intersection with $\Pi$, then, by the \emph{theorem of local existence and uniqueness}, there exists a unique solution to the Cauchy problem in \eqref{eq_CauchyProblem_Adimensionalized}.

\subsection{``One-disclination problem'': the effect of the boundary}\label{sec_DynamicsOneDisclination}
We now discuss the dynamics of a single disclination, initially located in  $\Xi_{0}\in B_{1}(0)$. Hence, upon specializing \eqref{eq_CauchyProblem_Adimensionalized} to the case $N = 1$, we find
\begin{linenomath}
    \begin{equation}
        \begin{cases}
            \dot{\Xi}(T) = 2s^{2}(1-|\Xi(T)|^{2})\Xi(T), &T>0,\\
            \Xi(0) = \Xi_{0},
        \end{cases}
        \label{eq_eccentriclDisclination_4a}
   \end{equation}
\end{linenomath}
which has the structure of a Cauchy problem with a \emph{logistic} differential equation. To solve it, it is convenient to redefine $\Xi$ as the complex-valued function of time $\Xi(T) \equiv \rho(T) \mathrm{exp}(i\phi(T))$, where $\rho \colon [0,+\infty[\,\rightarrow [0,+\infty[$ and $\phi \colon [0,+\infty[\,\rightarrow \mathbb{R}$. Thus,  \eqref{eq_eccentriclDisclination_4a} can be rewritten as 
\begin{linenomath}
        \begin{align}
            &\begin{cases}
                \dot{\rho}(T) = 2s^{2}(1-\rho(T)^{2})\rho(T), &T>0,\\
                \rho(0) = \rho_{0},
            \end{cases} && 
            \begin{cases}
                \dot{\phi}(T) = 0, &T>0,\\
                \phi(0) = \phi_{0}.
            \end{cases}
            \label{eq_AnalyticalSolution_singleDisclination_1}
        \end{align}
\end{linenomath}
Note that assuming $\Xi_{0}$ to be a point of the open ball $B_{1}(0)$ implies the restrictions $\rho_{0}\in[0,1[\,,$ which prevents $\rho_{\mathrm{st}}(T)=1$ from being a stationary solution to the Cauchy problem under study. In fact, the null function is the only stationary solution to \eqref{eq_AnalyticalSolution_singleDisclination_1}$_{1}$, obtained upon choosing  $\rho_{0}=0$. However, if $\rho_{0}\in\,]0,1[\,$, the solutions to the  Cauchy problems in \eqref{eq_AnalyticalSolution_singleDisclination_1} can be determined analytically, and, upon setting $\mu_{0}\defeq (1-\rho_{0}^{2})/\rho_{0}^{2} > 0$, they read
\begin{linenomath}
        \begin{align}
            &\rho(T) = \frac{1}{\sqrt{1+\mu_{0}\exp(-4s^{2}T)}},
            && \phi(T) = \phi_{0},
            && \forall\,T\in [0,+\infty[\,.
            \label{eq_AnalyticalSolution_singleDisclination_2}
        \end{align}
\end{linenomath}
This result states that the dynamics of the disclination $D=(s,\Xi)$, i.e., the evolution in time of $\Xi(T)$ for a given value of $s$, is solely \emph{radial}, as one expects from the fact that the right-hand side of \eqref{eq_eccentriclDisclination_4a}, i.e., the non-dimensional force, is  \emph{radial}. Moreover, we have that $\rho(T)\rightarrow 1$ if  $T\rightarrow+\infty$. Hence, we have that a disclination, initially located in a point different from the origin, will reach the boundary of $B_{1}(0)$ in the limit for $T\to+\infty$. We emphasize that this behavior contrasts with the one of one screw dislocation 
in $B_{1}(0)\backslash\{0\}$, which reaches the boundary in finite time \cite{Blass2015a,HM2017}. 

\subsection{``Two-disclination problem'': the interaction between disclinations}
\label{sec_TwoDisclinations}
We consider now the dynamics of two disclinations in $B_{1}(0)$, namely $D_{1} = (s_{1}, \Xi_{1})$ and $D_{2} = (s_{2}, \Xi_{2})$, thereby specializing \eqref{eq_CauchyProblem_Adimensionalized} to $N=2$, i.e.,
\begin{linenomath}
    \begin{align}
        \begin{cases}
            \dot{\Xi}_{1}(T) = \hat{F}_{1}(\Xi_{1}(T), \Xi_{2}(T)), & \dot{\Xi}_{2}(T) = \hat{F}_{2}(\Xi_{1}(T), \Xi_{2}(T)), \quad T>0,\\
            \Xi_{1}(0) = \Xi_{10}, &\Xi_{2}(0) = \Xi_{20},
        \end{cases}
        \label{eq_CauchyProblem_Adimensionalized_N2}
    \end{align}
\end{linenomath}
where the non-dimensional forces $\hat{F}_{1} \coloneqq -\nabla_{\Xi_{1}}\hat{H}$ and $\hat{F}_{2} \coloneqq -\nabla_{\Xi_{2}}\hat{H}$ read
\begin{linenomath}
    \begin{subequations}
        \begin{align}
            &\hat{F}_{1}(\Xi_{1},\Xi_{2}) = \hat{F}_{1}^{(1)}(\Xi_{1}) + \hat{F}_{1}^{(2)}(\Xi_{1}, \Xi_{2}) + \hat{F}_{1}^{(3)}(\Xi_{1}, \Xi_{2}),
            \label{eq_forces_ab_a}\\
            &\hat{F}_{2}(\Xi_{1},\Xi_{2}) = \hat{F}_{2}^{(1)}(\Xi_{2}) + \hat{F}_{2}^{(2)}(\Xi_{1}, \Xi_{2}) + \hat{F}_{2}^{(3)}(\Xi_{1}, \Xi_{2}).
            \label{eq_forces_ab_b}
        \end{align}
    \end{subequations}
\end{linenomath}
Similarly to the general case described above for $N$ disclinations, the force $\hat{F}_{1}$ acting on the disclination $D_{1}$ features three contributions: $\hat{F}_{1}^{(1)}$ is formally the same \emph{radial} force that would act on $D_{1}$ if the disclination $D_{2}$ were absent (see Section \ref{sec_DynamicsOneDisclination}); $\hat{F}_{1}^{(2)}$ is the action, modulated by $D_{2}$, that the boundary of $B_{1}(0)$ exerts on $D_{1}$; $\hat{F}_{1}^{(3)}$ is the \emph{mutual} interaction between $D_{1}$ and $D_{2}$. In particular, $\hat{F}_{1}^{(2)}$ is \emph{radial}, and $\hat{F}_{1}^{(3)}$ is directed along the segment connecting $\Xi_{1}$ and $\Xi_{2}$, while their signs depend solely on the product of the Frank angles $s_{1}s_{2}$\footnote{Indeed, within the range $0<\Phi_{12}<1$, it yields that $1-\Phi_{12}>0$ and $1-\Phi_{12}+\log\Phi_{12}<0$.}. The same discussion can be done also for the force featuring in \eqref{eq_forces_ab_b}.

If $s_{1}$ and $s_{2}$ have the same sign, i.e., if $s_{1}s_{2}>0$, then, for each disclination, the second force tends to move the disclination radially towards the boundary, in synergy with the first force, and the third force describes a \emph{repulsion} between the two disclinations. On the other hand, if $s_{1}$ and $s_{2}$ have opposite sign, i.e., if $s_{1}s_{2}<0$, the second force operates against the first one by pushing the disclinations towards the center, and the third force describes an \emph{attraction} between the two disclinations.

A consequence of the discussion reported above is that disclinations in an elastic medium behave qualitatively as \emph{charged particles} (see the notion of topological charges \cite{Qi2009}), with the Frank angles playing the role of electric charges. Thus, two disclinations with Frank angles of opposite sign attract each other, whereas two disclinations with Frank angles of the same sign repel each other. Moreover, if two disclinations come ``very close'' to each other (in a sense that will be rigorously explained later), then we say that the two disclinations \emph{collide} \cite{Blass2015a}, or \emph{merge} \cite{Qi2009}. Note that two disclinations colliding with each other do not superpose, since the superposition condition $\Xi_{1}=\Xi_{2}$ can only be reached in infinite time, provided certain conditions on the Frank angles of the disclinations are met. For this reason, if such conditions are respected, the superposition can be understood as the limit of a collision for time going towards infinity. In other words, even though two different disclinations may tend to intersect, their positions must always comply with the inequality $|\Xi_{1}(T)-\Xi_{2}(T)|>0$ for all times. Still, upon introducing a \emph{critical distance} $\varepsilon_{\mathrm{c}}>0$ between the centers of the two disclinations, there exists a critical time $T_{\mathrm{c}}$, which we refer to as \emph{collision time}, such that, for $T>T_{\mathrm{c}}$, and within a certain time range, it may occur that $0<|\Xi_{1}(T)-\Xi_{2}(T)|\leq \varepsilon_{\mathrm{c}}$. For the times at which this conditions 
are verified, we say that the two disclinations collide. Conversely, the collision ends if $|\Xi_{1}(T)-\Xi_{2}(T)|> \varepsilon_{\mathrm{c}}$.

To better understand the scenarios of a disclination approaching the boundary and of two superposing 
disclinations, we give the following Remarks:
\begin{remark}[A disclination goes to the boundary]\label{rem_LimitCase_1} Let us suppose 
that $D_{2}$ approaches the boundary, which amounts to taking the limit $|\Xi_{2}|\rightarrow 1^{-}$. Then, the force $\hat{F}_{2}$ tends to vanish and, accordingly, the evolution of $\Xi_{2}$ tends to stop, while the force $\hat{F}_{1}$ acting on $D_{1}$ tends to equal the first (radial) contribution of \eqref{eq_forces_ab_a}, i.e., $\hat{F}_{1}(\Xi_{1},\Xi_{2}) \sim 2s_{1}^{2}(1-|\Xi_{1}|^{2})\Xi_{1}$, since $D_{1}$ is the only disclination \emph{effectively} remaining in the system. It is important to highlight that, in fact, $\hat{F}_{1}$ becomes independent of the position taken by $D_{2}$ on the boundary. Expressed in the formalism of complex variables, $\hat{F}_{1}$ becomes independent of the ``phase'' of $\Xi_{2}$ (see Section \ref{sec_OneDisclination}).

\end{remark}

\begin{remark}[Superposition of two disclinations]\label{rem_LimitCase_2} Let us study the case in which two disclinations superpose at the same point $\Xi_{\mathrm{e}}\in B_{1}(0)$. This situation is handled by restarting from the definition of the energy of the system, which, in the limit $\Xi_{2}\rightarrow \Xi_{1}$, reads
\begin{linenomath}
        \begin{align}
            \check{H}(\Xi_{\mathrm{e}})\coloneqq\lim_{\Xi_{2}\rightarrow \Xi_{1}}\hat{H}(\Xi_{1},\Xi_{2})=\frac{1}{2}(s_{1}+s_{2})^{2}(1-|\Xi_{\mathrm{e}}|^{2})^{2}.
            \label{eq_energy_coexistence}
        \end{align}
\end{linenomath}
The limit $\Xi_{2}\rightarrow \Xi_{1}$ could describe some collisions for which $|\Xi_{1}(T)-\Xi_{2}(T)| \rightarrow 0$ as $T\rightarrow +\infty$, but, in general, the latter condition need not be true. Indeed, the operation $\Xi_{2}\rightarrow \Xi_{1}$ is rather understood here as virtually moving $D_{1}$ and $D_{2}$ until they superpose, but it does not have a dynamic meaning at this stage. In this regard, as anticipated in Section \ref{sec_unfixing}, two disclinations initially placed in the same position, so that $\Xi_{10} = \Xi_{20} = \Xi_{0}$, do not necessarily remain superposed for $T>0$, unless they are compelled to do so, or they annihilate each other, as is the case for $s_{1}+s_{2} = 0$. Indeed, if $s_{1}+s_{2} = 0$, then $\Xi_{1}(T) = \Xi_{2}(T) = \Xi_{0}$ for all $T>0$, independently from the value of $\Xi_{0}\in B_{1}(0)$. This can be proved by direct inspection of the forces in \eqref{eq_forces_ab_a} and \eqref{eq_forces_ab_b} for this very particular situation. However, if $s_{1}+s_{2} \neq 0$, each $\Xi_{0} \in B_{1}(0)\setminus\{0\}$ generates interactions with $\partial B_{1}(0)$ that break the initial superposition, thereby separating the two disclinations dynamically. Yet, although maintaining the conditions $s_{1}+s_{2} \neq 0$ and $\Xi_{0} \in B_{1}(0)\setminus\{0\}$, we may enforce a \emph{constraint} restricting the two disclinations to remain superposed at all times. This is useful, for instance, to quantify the reactive force necessary to maintain the superposition also for $T>0$. To account for this constrained setting, we introduce the \emph{holonomic} constraint  $\hat{C}_{\varepsilon}(\Xi_{1}, \Xi_{2}) \coloneqq \Xi_{1}-\Xi_{2} -\varepsilon = 0$, with $\varepsilon$ being a constant vector with arbitrarily small magnitude, and the \emph{augmented} Lagrangian function \cite{Lanczos1970a}
\begin{linenomath}
    \begin{align}
        \mathcal{L}_{\varepsilon}\equiv \hat{\mathcal{L}}_{\varepsilon}(\Xi_{1},\Xi_{2};\lambda) \coloneqq -\hat{H}(\Xi_{1}, \Xi_{2}) + \lambda\cdot\hat{C}_{\varepsilon}(\Xi_{1}, \Xi_{2}),
        \label{eq_Lagrangian}
    \end{align}
\end{linenomath}
with $\lambda$ being the two-dimensional Lagrange multiplier associated with the constraint. Note that the smallness parameter $\varepsilon$ is introduced to prevent the logarithm in the forces presented above to be ill-defined, and that the ``true'' Lagrangian of the system is given by $L\equiv \hat{L}(\Xi_{1},\Xi_{2})\coloneqq-\hat{H}(\Xi_{1},\Xi_{2})$ in the present framework (no inertial forces are attributed to the disclinations).
Because of the constraint in \eqref{eq_Lagrangian}, the dynamic problem in \eqref{eq_CauchyProblem_Adimensionalized_N2} changes into
\begin{linenomath}
\begin{subequations}
    \begin{align}
            \dot{\Xi}_{1}(T) &= \nabla_{\Xi_{1}}\hat{\mathcal{L}}_{\varepsilon}(\Xi_{1}(T),\Xi_{2}(T);\lambda(T)) = -\nabla_{\Xi_{1}}\hat{H}(\Xi_{1}(T), \Xi_{2}(T)) + \lambda(T),
            \label{eq_constraints_a}\\
            \dot{\Xi}_{2}(T) &= \nabla_{\Xi_{2}}\hat{\mathcal{L}}_{\varepsilon}(\Xi_{1}(T),\Xi_{2}(T);\lambda(T)) = -\nabla_{\Xi_{2}}\hat{H}(\Xi_{1}(T), \Xi_{2}(T)) - \lambda(T),
            \label{eq_constraints_b}\\
            0 &= \nabla_{\lambda}\hat{\mathcal{L}}_{\varepsilon}(\Xi_{1}(T),\Xi_{2}(T);\lambda(T)) = \Xi_{1}(T)-\Xi_{2}(T) - \varepsilon,
            \label{eq_constraints_c}
    \end{align}
\end{subequations}
\end{linenomath}
which holds for $T>0$. By differentiating \eqref{eq_constraints_c} with respect to time, the system in \eqref{eq_constraints_a}, \eqref{eq_constraints_b}, and \eqref{eq_constraints_c} can be rewritten in the decoupled form
\begin{linenomath}
\begin{subequations}
    \begin{align}
            &\dot{\Xi}_{1}(T) = -\tfrac{1}{2}\big[\nabla_{\Xi_{1}}\hat{H}\big(\Xi_{1}(T), \Xi_{1}(T)-\varepsilon\big) + \nabla_{\Xi_{2}}\hat{H}\big(\Xi_{1}(T), \Xi_{1}(T)-\varepsilon\big)\big],
            \label{eq_constraint_FinalEquations_a}\\
            &\dot{\Xi}_{2}(T) = \dot{\Xi}_{1}(T).
            \label{eq_constraint_FinalEquations_b}\\
            &\lambda(T) = -\nabla_{\Xi_{2}}\hat{H}\big(\Xi_{1}(T), \Xi_{1}(T)-\varepsilon\big) - \dot{\Xi}_{1}(T).
            \label{eq_constraint_FinalEquations_c}
    \end{align}
\end{subequations}
\end{linenomath}
The solution to   \eqref{eq_constraint_FinalEquations_a}, \eqref{eq_constraint_FinalEquations_b}, and \eqref{eq_constraint_FinalEquations_c} that comply with the initial conditions $\Xi_{1}(0) = \Xi_{10}$ and $\Xi_{2}(0) = \Xi_{10} - \varepsilon$ is the sequence of triples $(\Xi_{1}(T,\varepsilon),\Xi_{2}(T,\varepsilon),\lambda(T,\varepsilon))$, which, in the limit $\varepsilon\to 0$, converges to the solution of the problem
\begin{linenomath}
    \begin{align}
        \begin{cases}
        \dot{\Xi}_{\mathrm{e}}(T) = \hat{F}_{\mathrm{e}}(\Xi_{\mathrm{e}}(T)) = 2s_{\mathrm{e}}^{2}(1-|\Xi_{\mathrm{e}}(T)|^{2})\Xi_{\mathrm{e}}(T), &T>0,\\
        \lambda(T) = (s_{2}^{2}-s_{1}^{2})(1-|\Xi_{\mathrm{e}}(T)|^{2})\Xi_{\mathrm{e}}(T), &T>0,\\
        \Xi_{\mathrm{e}}(0) = \Xi_{\mathrm{e}0},
        \end{cases} &&s_{\mathrm{e}} \coloneqq \frac{s_{1}+s_{2}}{\sqrt{2}}.
        \label{eq_effective_equations}
    \end{align}
\end{linenomath}
Thus, if the constraint of ``persisting superposition'' is fulfilled, the two disclinations follow together the same kind of dynamics addressed in Section \ref{sec_OneDisclination}, as if they were a single disclination, initially positioned in $\Xi_{\mathrm{e}0}$, and with an ``effective'' Frank angle equal to $s_{\mathrm{e}} \coloneqq (s_{1}+s_{2})/\sqrt{2}$, which can be zero or different from zero. However, the force on the right-hand side of \eqref{eq_effective_equations}$_{1}$ is different from the one that would stem from the differentiation of the energy $\check{H}$ in \eqref{eq_energy_coexistence} with respect to $\Xi_{\mathrm{e}}$. The ``price to pay'' for maintaining this dynamics is given by the reaction forces $\lambda(T)$ and $-\lambda(T)$, which are characterized by the following property: they have the same functional dependence on $\Xi_{\mathrm{e}}$ as the force in \eqref{eq_effective_equations}$_{1}$, but they depend on the difference between the squares of the Frank angles of the two disclinations. Hence, the reactive forces vanish identically both when the Frank angles are equal to each other and when they are opposite to each other. This means that constraining the two disclinations to persist in their initial superposed configuration yields null reaction forces for the subsequent instants of time, since the two disclinations remain together naturally. This is because, in the just discussed situation, the superposition of the two disclinations annihilates the attractive or repulsive forces that they mutually exchange, and only modulates the interactions with the boundary. On the contrary, the reaction forces are non-null for $|s_{1}|\neq |s_{2}|$, and their magnitudes increase with the absolute value $|s_{2}^{2}-s_{1}^{2}|$, since the ``price'' for keeping them superposed increases with their tendency to separate.

\end{remark}

\paragraph{Two disclinations with zero total Frank angle} To better understand the behaviour of the ``two-disclination problem'', we now consider the case in which $s_{1} = -s_{2} = s$, and the disclinations are initially placed in $B_{1}(0)$ symmetrically with respect to the origin. In this case, \eqref{eq_CauchyProblem_Adimensionalized_N2} reduces to a Cauchy problem with a single ordinary differential equation in the non-dimensional distance $\Delta \coloneqq |\Xi_{1}-\Xi_{2}|$ between the two disclinations, and with initial condition $\Delta(0) = \Delta_{0} \coloneqq |\Xi_{10}-\Xi_{20}|$, i.e.,
\begin{linenomath}
    \begin{align}
        \begin{cases}
            \dot{\Delta}(T) = \hat{G}(\Delta(T)), &T>0,\\
            \Delta(0) = \Delta_{0},
        \end{cases}
        \label{eq_CauchyProblem_SymmetricScenario}
    \end{align}
\end{linenomath}
where the forcing term $\hat{G}(\Delta)\coloneqq \hat{F}_{1}(\Xi_{1}, \Xi_{2}) - \hat{F}_{2}(\Xi_{1}, \Xi_{2})$ is given by 
\begin{linenomath}
\begin{align}
    \hat{G}(\Delta) \coloneqq 4s^{2}\Delta\bigg[\frac{4-\Delta^{2}}{4+\Delta^{2}} + 2\log\frac{4\Delta}{4+\Delta^{2}}\bigg], &&\Delta\in\,]0,2[\,,
    \label{eq_CauchyProblem_SymmetricScenario_Force}
\end{align}
\end{linenomath}
and is obtained because, in the symmetric condition under investigation, it holds that $\Xi_{1}=(\frac{1}{2}\Delta,0)$ and $\Xi_{2}=(-\frac{1}{2}\Delta,0)$.

The function $\hat{G}$ in \eqref{eq_CauchyProblem_SymmetricScenario_Force} admits only one zero in the open interval $\,]0,2[\,$, computed numerically, and approximately given by $\Delta_{\mathrm{eq}} \simeq 0.8$. Since $\hat{G}(\Delta_{\mathrm{eq}})=0$, $\Delta_{\mathrm{eq}}$ is the sole equilibrium configuration for the system. This configuration is \emph{unstable} and divides the range of admissible values of $\Delta$ into two basins of attraction, as shown in Figure \ref{fig_SuperSymmetricScenario}. The instability of $\Delta_{\mathrm{eq}}$ descends from the fact that if we choose $\Delta_{0} < \Delta_{\mathrm{eq}}$ in the Cauchy problem \eqref{eq_CauchyProblem_SymmetricScenario}, then it holds that $\Delta(T) \to 0^{+}$ for $T\to+\infty$, which means that, if the initial distance between the two disclinations is smaller than the equilibrium one, it tends to decrease indefinitely, so that the disclinations collide and ``eventually'' superpose in the origin. Moreover, if $\Delta_{0} > \Delta_{\mathrm{eq}}$, then $\Delta(T) \to 2^{-}$ for $T\rightarrow +\infty$, i.e., the disclinations increase their reciprocal distance with respect to the initial one, and move towards the boundary.

What has been described so far can be formalized by taking the principal parts of $\hat{G}(\Delta)$ with respect to $\Delta \to 0^{+}$ and $\Delta \to 2^{-}$, which read 
\begin{linenomath}
    \begin{subequations}
        \begin{align}
            &\hat{G}(\Delta)/s^{2} \sim 8\Delta \, \log\!\Delta, && \Delta \to 0^{+},
            \label{eq_TwoDisclinations_8a}\\
            &\hat{G}(\Delta)/s^{2} \sim 4(2-\Delta), && \Delta \to 2^{-},
            \label{eq_TwoDisclinations_8b}
        \end{align}
   \end{subequations}
\end{linenomath}
and solving the corresponding approximated Cauchy problems in small neighborhoods $I^{+}(0)$ and $I^{-}(2)$ of the values $\Delta = 0$ and $\Delta = 2$. Thus, we have that 
\begin{linenomath}
    \begin{subequations}\label{eq_4.18}
        \begin{align}
            &\begin{cases}
                \dot{\Delta}(T) = 8s^{2}\Delta(T)\log\!\Delta(T),\\
                \Delta(0) = {\Delta}_{0} \in I^{+}(0),
            \end{cases}&& \Rightarrow \Delta(T) = \exp[(\log\Delta_{0})\exp(+8s^{2}T)],
            \label{eq_TwoDisclinations_10a}\\
            &\begin{cases}
                \dot{\Delta}(T) = 4s^{2}(2-\Delta(T)),\\
                \Delta(0) = {\Delta}_{0} \in I^{-}(2),
            \end{cases} &&\Rightarrow \Delta(T) = 2-(2-\Delta_{0})\exp(-4s^{2}T).
            \label{eq_TwoDisclinations_10b}
        \end{align}
   \end{subequations}
\end{linenomath}
The results in 
\eqref{eq_TwoDisclinations_10a} and \eqref{eq_TwoDisclinations_10b} 
imply that the values $\Delta = 0$ and $\Delta = 2$ are attractors for the system under study. In particular, although $\Delta = 0$ is not in the domain of $\hat{G}$, it describes the case of two superposing  disclinations with opposite Frank angles, and is consistent with the predictions of Remark \ref{rem_LimitCase_2} for $s_{\mathrm{e}}=0$ and $\Xi_{\mathrm{e}0}=0$ (the latter condition preserves the compliance with the geometric constraint that the two disclinations are placed symmetrically with respect to the origin).


%
\begin{figure}[htbp]
    \centering
    \begin{subfigure}{0.6\textwidth}
    \centering
    \includegraphics[width=\textwidth]{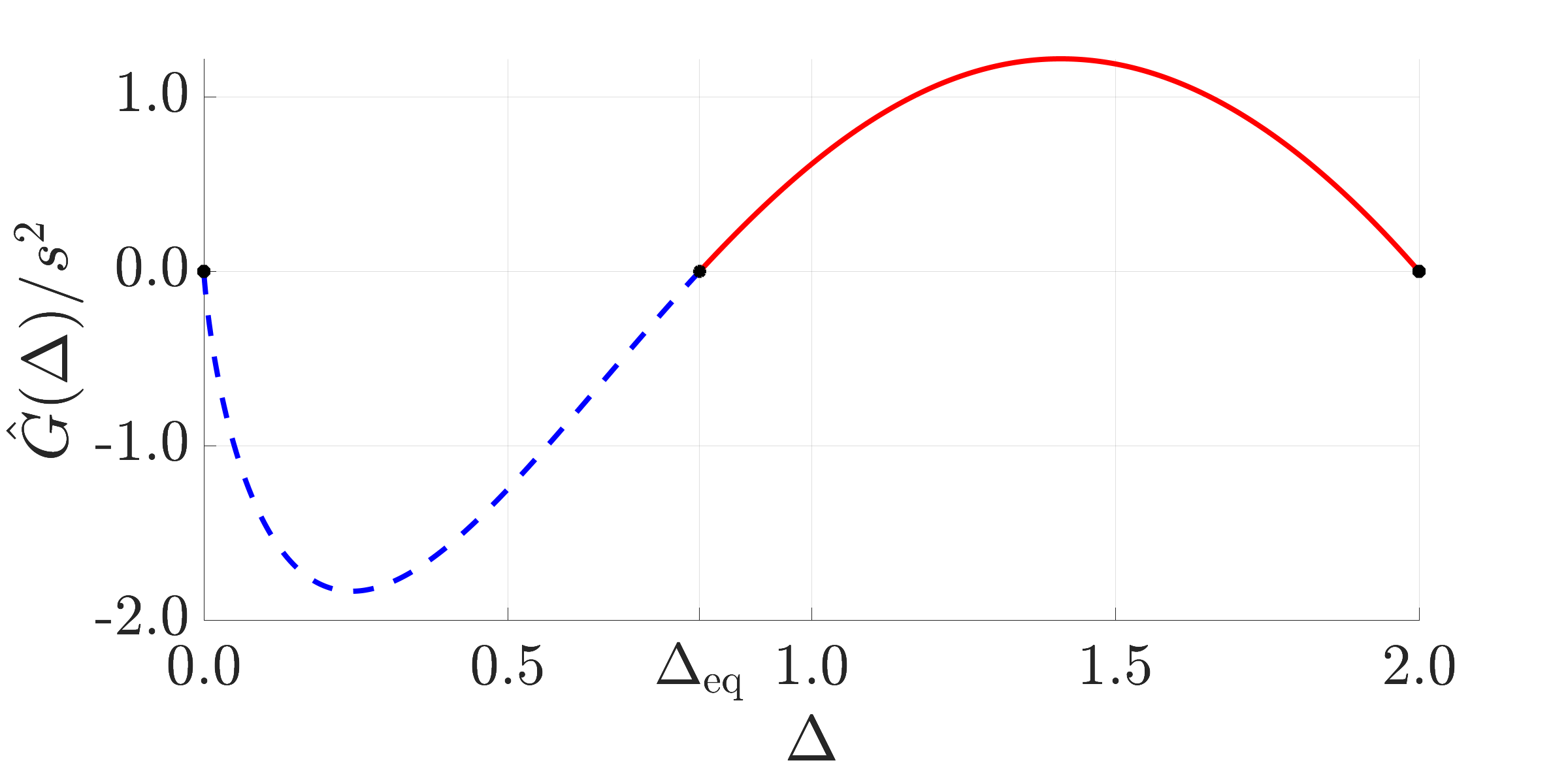}
    \end{subfigure}
    \hfill
    \begin{subfigure}{0.35\textwidth}
    \centering
    \includegraphics[width=\textwidth]{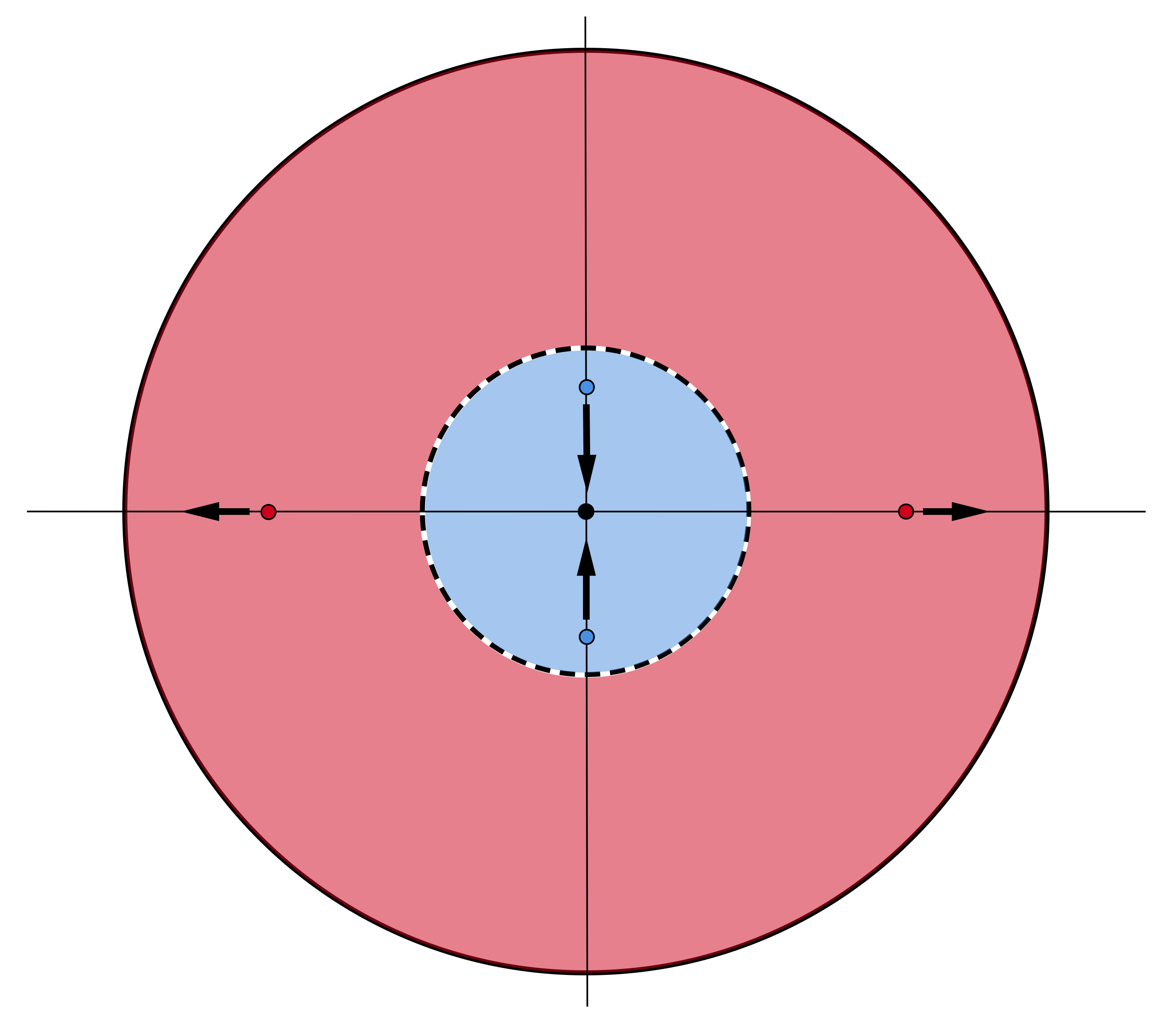}
    \end{subfigure}
    \caption{(Left) Plot of $\hat{G}(\Delta)/s^{2}$ in \eqref{eq_CauchyProblem_SymmetricScenario}, and the basins of attraction for $\Delta = 0$ (dashed line) and $\Delta = 2$ (solid line). (Right) Sketch of the motion of the disclinations, which come closer for $\Delta_{0}<\Delta_{\mathrm{eq}}$, and move apart for  $\Delta_{0}>\Delta_{\mathrm{eq}}$.}
    \label{fig_SuperSymmetricScenario}
\end{figure}

Even more importantly, from the just discussed situation (two disclinations with opposite Frank angles and symmetrically placed) we deduce that the value $\Delta_{\mathrm{eq}}$ is the system's \emph{critical distance} such that two behaviors are observed: On the one hand, for $\Delta_{0}<\Delta_{\mathrm{eq}}$, the attractive forces between the disclinations are predominant with respect to the effect of the boundary, and, thus, as one would expect, the disclinations tend to come indefinitely closer to each other; on the other hand, for $\Delta_{0}>\Delta_{\mathrm{eq}}$, the presence of the boundary overcomes the attractive forces, and, even though the disclinations have opposite Frank angles, they move apart.

\section{Modeling the crystalline structure}\label{sec_glideDirections}
This section is devoted to studying the dynamical system in \eqref{eq_CauchyProblem_Adimensionalized} 
in 
the presence of preferred directions for the motion, in a similar fashion to the concept of \emph{glide directions}, as introduced in \cite{Cermelli1999a} and studied in \cite{Blass2015a}. 
While there is no strong consensus supported by experimental evidence 
that this be the case 
\cite{RV1983},
there is nonetheless evidence that their dynamics is indeed influenced by these slip systems.
In metal alloys and elastic crystals, the mechanism of twin deformation has been modeled in terms of the motion of disclination dipoles. It is well known that the distortion induced by disclination dipoles is effectively comparable to that induced by a single edge dislocation \cite{Cesana2024a,Eshelby66}. Since dislocation kinetics is practically constrained by the glide directions of slip systems, it seems reasonable to consider the motion of these disclination dipoles as also constrained by a system of glide directions.
Even in the context of graphene, there are indications that disclination motion depends on preferred directions within the hexagonal lattice and specific angles. The mechanism of nucleation of single disclinations and their kinetics in graphene is postulated in \cite{AJE2021}
to involve initial 90-degree bond rotations, followed by the subsequent migration of dislocation-disclination dipole systems (for further work in this direction see e.g. \cite{Yavari_MRC, Yavari_PRSA}).

In this section, we denote by $\boldsymbol{\Xi}=(\Xi_1,\ldots,\Xi_N)\in [B_1(0)]^N$ the configuration of the $N$ disclinations and we shall abandon the ``hat'' notation; 
recalling \eqref{eq_ExistenceTheorem_2}, the system in \eqref{eq_CauchyProblem_Adimensionalized} reads then (with $\boldsymbol{F}(\boldsymbol{\Xi}) \coloneqq (F_{1}(\boldsymbol{\Xi}),\ldots,F_{N}(\boldsymbol{\Xi})) \in \mathbb{R}^{2N}$)
\begin{equation}\label{eq:100}
\begin{cases}
    \dot{\boldsymbol{\Xi}}(T)=\boldsymbol{F}(\boldsymbol{\Xi}(T)), & T>0, \\
    \boldsymbol{\Xi}(0)=\boldsymbol{\Xi}_{0}.
\end{cases}
\end{equation}

Let $\Gcal$ 
be a set of  $M\in\mathbb{N}$ unit vectors, 
with the requirements that $\mathrm{Span}\,\Gcal = \mathbb{R}^{2}$ and that $g\in\Gcal \iff -g\in\Gcal$. 
The 
directions $g_k(\boldsymbol{\Xi})\in\Gcal$ along which the disclination $D_k=(s_k,\Xi_k)$ moves are determined as the ones along which the force~$F_k$ is the most aligned, namely
%
\begin{equation}\label{g_k}
g_k(\boldsymbol{\Xi})\in \arg\max_{g\in\Gcal}\{F_{k}(\boldsymbol{\Xi})\cdot g\}\,;
\end{equation}
we let $ \Gcal_k(\boldsymbol{\Xi})\coloneqq \{g_k(\boldsymbol{\Xi})\in\Gcal : \text{\eqref{g_k} holds}\}$.
Finally, we define
\begin{linenomath}
    \begin{align}
            \mathcal{F}_{k}(\boldsymbol{\Xi}) = \big\{(F_{k}(\boldsymbol{\Xi})\cdot g_k(\boldsymbol{\Xi}))g_k(\boldsymbol{\Xi})
            \,\big|\,g_k(\boldsymbol{\Xi})\in \Gcal_k(\boldsymbol{\Xi})\big\}.
            \label{eq_GlideDirections_1}
        \end{align}
\end{linenomath}
%
For geometric reasons, there are only three possible scenarios:
\begin{enumerate}
    \item[(1)] If $F_{k}(\boldsymbol{\Xi}) = 0$ (the force on disclination $D_k$ vanishes), then 
    $\Gcal_{k}(\boldsymbol{\Xi}) = \Gcal$, and trivially 
    $\mathcal{F}_{k}(\boldsymbol{\Xi}) = \{0\}$.
    \item[(2)] If $F_{k}(\boldsymbol{\Xi}) \neq 0$ and there is only one direction maximizing $F_{k}(\boldsymbol{\Xi})\cdot g$ in \eqref{g_k}, namely 
    $\Gcal_{k}(\boldsymbol{\Xi}) = \{g_{k}(\boldsymbol{\Xi})\}$, then $\mathcal{F}_{k}(\boldsymbol{\Xi}) = \{(F_{k}(\boldsymbol{\Xi})\cdot g_{k}(\boldsymbol{\Xi}))g_{k}(\boldsymbol{\Xi})\}$.
    \item[(3)] If $F_{k}(\boldsymbol{\Xi}) \neq 0$ and there exist two different 
    directions both maximizing $F_{k}(\boldsymbol{\Xi})\cdot g$ in \eqref{g_k}, namely 
    $\Gcal_{k}(\boldsymbol{\Xi}) = \{g_{k}^{-}(\boldsymbol{\Xi}), g_{k}^{+}(\boldsymbol{\Xi})\}$, then $\mathcal{F}_{k}(\boldsymbol{\Xi}) = \{(F_{k}(\boldsymbol{\Xi})\cdot g_{k}^{\pm}(\boldsymbol{\Xi}))g_{k}^{-}(\boldsymbol{\Xi}), (F_{k}(\boldsymbol{\Xi})\cdot g_{k}^{\pm}(\boldsymbol{\Xi}))g_{k}^{+}(\boldsymbol{\Xi})\}$. This case occurs when the force $F_k(\boldsymbol{\Xi})$ bisects the angle formed by $g_k^-(\boldsymbol{\Xi})$ and $g_k^+(\boldsymbol{\Xi})$.
\end{enumerate}
These three cases can be resumed in the following formula:
\begin{linenomath}
        \begin{align}
            \mathcal{F}_{k}(\boldsymbol{\Xi}) = 
\begin{cases}
    \{0\}, &\text{if $F_{k}(\boldsymbol{\Xi}) = 0$},\\
    \{(F_{k}(\boldsymbol{\Xi})\cdot g_{k}(\boldsymbol{\Xi}))g_{k}(\boldsymbol{\Xi})\}, &\text{if $F_{k}(\boldsymbol{\Xi}) \neq 0$ and $\Gcal_k(\boldsymbol{\Xi}) = \{g_{k}(\boldsymbol{\Xi})\}$},\\
    \{(F_{k}(\boldsymbol{\Xi})\cdot g_{k}^{\pm}(\boldsymbol{\Xi}))g_{k}^{\pm}(\boldsymbol{\Xi})\}, &\text{if $F_{k}(\boldsymbol{\Xi}) \neq 0$ and $\Gcal_{k}(\boldsymbol{\Xi}) = \{g_{k}^{\pm}(\boldsymbol{\Xi})\}$},
\end{cases}
            \label{eq_GlideDirections_3}
        \end{align}
\end{linenomath}
and \eqref{eq:100}
takes the form of the following differential inclusion:
\begin{linenomath}
\begin{equation}\label{eq_GlideDirections_2}
        \begin{cases}
            \dot{\boldsymbol{\Xi}}(T) \in \boldsymbol{\mathcal{F}}(\boldsymbol{\Xi}(T)) & T>0,\\
            \boldsymbol{\Xi}(0) = \boldsymbol{\Xi}_{0}\,,
        \end{cases}
   \end{equation}
\end{linenomath}
where 
$\boldsymbol{\mathcal{F}}(\boldsymbol{\Xi})\defeq \mathcal{F}_{1}(\boldsymbol{\Xi})\times\cdots\times\mathcal{F}_{N}(\boldsymbol{\Xi})\subset \mathbb{R}^{2N}$;
notice that \eqref{eq_GlideDirections_2} is a dynamics in $\mathbb{R}^{2N}$, and it is indeed a differential inclusion when scenario (3) occurs for at least one $k\in\{1,\ldots,N\}$.

The set-valued function $\boldsymbol{\mathcal{F}}$
is defined on the set $
[B_{1}(0)]^{N}\setminus\Pi$, which makes each $F_k$ well defined; 
nonetheless, recalling the arguments in 
Section \ref{sec_nondim} and Remark \ref{rem_LimitCase_1}, 
it is possible to extend each $F_k$ by continuity to the whole closed unit disk, so that we will consider $\mathrm{dom}\,\boldsymbol{\mathcal{F}}=
\big[\overline{B}_{1}(0)\big]^{N}$; 
the set-valued function $\boldsymbol{\mathcal{F}}$ turns out to be upper semi-continuous (in the sense of \cite[p.~65]{Filippov}; see \cite[Lemma~2.11]{Blass2015a} for a proof).
The theory of differential inclusions \cite{Filippov} provides a suitable notion of solution to \eqref{eq_GlideDirections_2}.
\begin{definition}\label{def_sol_diff_inclusion}
A solution to \eqref{eq_GlideDirections_2} is
an absolutely continuous curve $T\mapsto \boldsymbol{\Xi}(T)\in [B_1(0)]^N$ such that 
\begin{linenomath}
    \begin{equation}\label{eq_GlideDirections_4}
        \begin{cases}
        \dot{\boldsymbol{\Xi}}(T) \in \mathrm{co}\,\boldsymbol{\mathcal{F}}(\boldsymbol{\Xi}(T)), & T>0, \\
        \boldsymbol{\Xi}(0) = \boldsymbol{\Xi}_{0}\in[B_1(0)]^N\setminus\Pi, 
        \end{cases}
        \end{equation}
\end{linenomath}
where $\mathrm{co}\,\boldsymbol{\mathcal{F}}$ is the convex hull of the set $\boldsymbol{\mathcal{F}}$.
\end{definition}
%
%
Our case is analogous to that of  \cite[Lemma~2.10]{Blass2015a}, therefore $\mathrm{co}\,\boldsymbol{\mathcal{F}} = (\mathrm{co}\,\mathcal{F}_{1}(\boldsymbol{\Xi}))\times\cdots\times(\mathrm{co}\,\mathcal{F}_{N}(\boldsymbol{\Xi}))\subset \mathbb{R}^{2N}$ for all $\boldsymbol{\Xi}\in [\overline{B}_{1}(0)]^{N}$, and, from \eqref{eq_GlideDirections_3}, we have that for each $k = 1,\ldots,N$ the term $\mathrm{co}\,\mathcal{F}_{k}(\boldsymbol{\Xi})$ is 
\begin{linenomath}
        \begin{align}
            \mathrm{co}\,\mathcal{F}_{k}(\boldsymbol{\Xi}) = 
\begin{cases}
    \{0\}, &\textrm{if $F_{k}(\boldsymbol{\Xi}) = 0$},\\
    \{(F_{k}(\boldsymbol{\Xi})\cdot g_{k}(\boldsymbol{\Xi}))g_{k}(\boldsymbol{\Xi})\}, &\textrm{if $F_{k}(\boldsymbol{\Xi}) \neq 0$ and $\Gcal_{k}(\boldsymbol{\Xi}) = \{g_{k}(\boldsymbol{\Xi})\}$},\\
\Sigma_{k}(\boldsymbol{\Xi}), &\textrm{if $F_{k}(\boldsymbol{\Xi}) \neq 0$ and $\Gcal_{k}(\boldsymbol{\Xi}) = \{g_{k}^{\pm}(\boldsymbol{\Xi})\}$},
\end{cases}
\label{eq_GlideDirections_5}
        \end{align}
\end{linenomath}
where $\Sigma_{k}(\boldsymbol{\Xi})$ is the segment joining $(F_{k}(\boldsymbol{\Xi})\cdot g_{k}^{\pm}(\boldsymbol{\Xi}))g_{k}^{-}(\boldsymbol{\Xi})$ and $(F_{k}(\boldsymbol{\Xi})\cdot g_{k}^{\pm}(\boldsymbol{\Xi}))g_{k}^{+}(\boldsymbol{\Xi})$.
We stress that scenario (3) must occur for having $\mathrm{co}\,\mathcal{F}_{k}(\boldsymbol{\Xi})=\Sigma_k(\boldsymbol{\Xi})$.

We can now state the local existence 
theorem.
\begin{theorem}[{see \cite[Theorem~2.13]{Blass2015a}}]
    Let $\boldsymbol{\mathcal{F}}\colon 
    [\overline{B}_1(0)]^N\to\mathcal{P}(\mathbb{R}^{2N})$ be defined as above 
    and let $\boldsymbol{\Xi}_0\in [{B}_1(0)]^N\setminus\Pi$. 
    Then there exists $T^*>0$ such that the \eqref{eq_GlideDirections_2} of motion admits a solution $[-T^*,T^*]\ni T\mapsto\boldsymbol{\Xi}(T)\in [{B}_1(0)]^N\setminus\Pi$ in the sense of Definition \ref{def_sol_diff_inclusion}.
\end{theorem}
To discuss the local uniqueness of the solution, we distinguish two possibilities: 
if either case (1) or case (2) occurs, namely if either $\mathrm{co}\,\mathcal{F}_{k}(\boldsymbol{\Xi}) = \{0\}$ or $\mathrm{co}\,\mathcal{F}_{k}(\boldsymbol{\Xi}) = \{(F_{k}(\boldsymbol{\Xi})\cdot g_{k}(\boldsymbol{\Xi}))g_{k}(\boldsymbol{\Xi})\}$, for all $k=1,\ldots,N$, then uniqueness is immediate owing to the Lipschitz continuity of the $\mathbb{R}^{2N}$-valued function $\boldsymbol{\mathcal{F}}$, 
whereas two distinct phenomena can occur in case (3), namely when there exists (at least one index) $k\in\{1,\ldots,N\}$ 
such that $\mathrm{co}\,\mathcal{F}_{k}(\boldsymbol{\Xi}) = \Sigma_{k}(\boldsymbol{\Xi})$. 
This last situation includes the phenomena that in the context of dislocation dynamics are called \emph{cross-slip} and \emph{fine cross-slip}: they were pictured in \cite{Cermelli1999a} and have been addressed mathematically in \cite{Blass2015a}.

To tackle uniqueness in the latter case, we need to look carefully at the condition leading to $\mathrm{co}\,\mathcal{F}_{k}(\boldsymbol{\Xi})=\Sigma_k(\boldsymbol{\Xi})$; from \eqref{eq_GlideDirections_5}, this occurs when $\mathrm{card}\,\mathcal{G}_k(\boldsymbol{\Xi})=2$, so that we define $\mathcal{A}_k\coloneqq\big\{\boldsymbol{\Xi}\in[\overline{B}_1(0)]^N:\mathrm{card}\,\mathcal{G}_k(\boldsymbol{\Xi})=2\big\}$ and $\mathcal{A}\coloneqq \cup_{k=1}^N \mathcal{A}_k$ (cf. \cite[formula (2.19)]{Blass2015a}). 
Given that in case (3) $F_{k}(\boldsymbol{\Xi})\neq0$, the sets $\mathcal{A}_k$ are identified by the condition $F_{k}(\boldsymbol{\Xi})\cdot g_0(\boldsymbol{\Xi})=0$, for the 
vector $g_0(\boldsymbol{\Xi})\coloneqq g_k^+(\boldsymbol{\Xi})-g_k^-(\boldsymbol{\Xi})$.
Invoking the analitycity results for biharmonic functions contained in \cite{Mus1919}, we obtain that $\mathcal{F}_{k}(\boldsymbol{\Xi})$ are analytic, so that all of the $\mathcal{A}_k$'s are locally smooth manifolds; this yields that the normal unit vector $\mathbf{N}(\boldsymbol{\Xi})$ is well defined for almost\footnote{Except at those points at which $\mathcal{A}_k$ and $\mathcal{A}_h$ intersect transversally for $h\neq k$. These points, however, belong to a set of codimension greater than or equal to $2$ in $\mathbb{R}^{2N}$.} all $\boldsymbol{\Xi}\in \mathcal{A}$ (and we assume, without loss of generality, that it points from the side of $\mathcal{A}_k$ where $g_k^-(\boldsymbol{\Xi})$ is the only direction satisfying \eqref{g_k} to the side of $\mathcal{A}_k$ where $g_k^+(\boldsymbol{\Xi})$ is the only direction satisfying \eqref{g_k}). 
Accordingly, given $\widehat{\boldsymbol{\Xi}}\in \mathcal{A}_k$ and $B$ a neighborhood of $\widehat{\boldsymbol{\Xi}}$, we denote by $B^+$ the ``half'' of $B$ on the side of $\mathcal{A}_k$ containing $\boldsymbol{N}(\widehat{\boldsymbol{\Xi}})$ and by $B^-$ the other half and we let 
\begin{linenomath}
\begin{align}
F_k^\pm(\boldsymbol{\Xi})\coloneqq ( 
F_k(\boldsymbol{\Xi})\cdot g_k^\pm(\widehat{\boldsymbol{\Xi}}))g_k^\pm(\widehat{\boldsymbol{\Xi}})\qquad \text{for $\boldsymbol{\Xi}\in B^\pm$} \nonumber
    \end{align}
\end{linenomath}
and we extend it smoothly to points $\boldsymbol{\Xi}\in B\cap\mathcal{A}_k$ (cf. \cite[formula (2.29)]{Blass2015a}); we let $\boldsymbol{F}^\pm(\boldsymbol{\Xi})$ be the corresponding forces.
Then, if at time $\widehat{T}$ it occurs that  $\boldsymbol{\Xi}(\widehat{T})\eqqcolon \widehat{\boldsymbol{\Xi}}\in \mathcal{A}_k$, and $\boldsymbol{\Xi}(T)\in B^-$ for $0<\widehat{T}-T$ sufficiently small, $\boldsymbol{\Xi}(T)\in B^+$ for $0<T-\widehat{T}$ sufficiently small, we have: 
\begin{enumerate}
    \item[(a)] \emph{cross-slip} (see \cite[Theorem~2.19]{Blass2015a}): the disclination $D_k$ switches from moving along direction $g_k^-(\boldsymbol{\Xi})$ to moving along direction $g_k^+(\boldsymbol{\Xi})$  if 
    \begin{linenomath}
        \begin{align}
    \boldsymbol{F}^-(\widehat{\boldsymbol{\Xi}})\cdot \boldsymbol{N}(\widehat{\boldsymbol{\Xi}})>0\qquad\text{and}\qquad  \boldsymbol{F}^+(\widehat{\boldsymbol{\Xi}})\cdot\boldsymbol{N}(\widehat{\boldsymbol{\Xi}})>0.\nonumber
        \end{align}
    \end{linenomath}
    \item[(b)] \emph{Fine cross-slip} (see \cite[Theorem~2.20]{Blass2015a}): the disclination $D_k$ switches from moving along direction $g_k^-(\boldsymbol{\Xi})$ to moving along a convex combination $g_k^0(\boldsymbol{\Xi})\notin\Gcal$ of directions $g_k^-(\boldsymbol{\Xi})$ and $g_k^+(\boldsymbol{\Xi})$ if 
    \begin{linenomath}
        \begin{align}
    \boldsymbol{F}^-(\widehat{\boldsymbol{\Xi}})\cdot \boldsymbol{N}(\widehat{\boldsymbol{\Xi}})>0\qquad\text{and}\qquad  \boldsymbol{F}^+(\widehat{\boldsymbol{\Xi}})\cdot\boldsymbol{N}(\widehat{\boldsymbol{\Xi}})<0.\nonumber
        \end{align}
    \end{linenomath}
    The force $\boldsymbol{F}^0(\boldsymbol{\Xi})\in\mathrm{co}\,\boldsymbol{\mathcal{F}}(\boldsymbol{\Xi})$ is determined by  $\boldsymbol{F}^0(\boldsymbol{\Xi})= (1-\alpha(\boldsymbol{\Xi}))\boldsymbol{F}^-(\boldsymbol{\Xi})+\alpha(\boldsymbol{\Xi})\boldsymbol{F}^+(\boldsymbol{\Xi})$, where
    \begin{linenomath}
        \begin{align}
            \alpha(\boldsymbol{\Xi})= \frac{\boldsymbol{F}^-(\boldsymbol{\Xi})\cdot \boldsymbol{N}(\boldsymbol{\Xi})}{\boldsymbol{F}^-(\boldsymbol{\Xi})\cdot \boldsymbol{N}(\boldsymbol{\Xi})-\boldsymbol{F}^+(\boldsymbol{\Xi})\cdot \boldsymbol{N}(\boldsymbol{\Xi})}.\nonumber
        \end{align}
    \end{linenomath}
    This motion occurs for $T>\widehat{T}$ and is confined to the set $\mathcal{A}_k$ (which makes the formula for $\alpha(\boldsymbol{\Xi})$ well defined) and abandons it (tangentially to $\mathcal{A}_k(\boldsymbol{\Xi})$) to continue in the direction $g_k^+(\boldsymbol{\Xi})$.
\end{enumerate}
Both scenarios (a) and (b) described above, as well as the result for right uniqueness in the case of single-valued forces $\mathcal{F}_k(\boldsymbol{\Xi})$ satisfy the conditions for right uniqueness of \cite[Theorem~2.17]{Blass2015a}. We refer the interested reader to \cite[Section~2]{Blass2015a} for the precise statements of these theorems and for some technical mathematical details that we omitted in our discussion, without consequences for the understanding of the dynamics.

\section{Benchmarks}
\label{sec_simulations}
In this section, we report on five \emph{benchmark} problems that we deem remarkable. 

\paragraph{\#1. Two disclinations with non-zero total Frank angle (Figure \ref{fig_dipolo})} We solve \eqref{eq_CauchyProblem_Adimensionalized_N2} in the case of the two disclinations with Frank angles of opposite sign and different magnitude, and symmetrically placed with respect to the origin of $B_{1}(0)$ at the initial reciprocal distance $\Delta_{0} = 0.6$. In the sequel, we denote by $D_{1}$ the disclination with positive Frank angle $s_{1}=+1$ and by $D_{2}$ the disclination with negative Frank angle $s_{2} = - 2$. In a relatively small right neighborhood of $T=0$, denoted by $\mathcal{U}^{+}(0)$, the magnitude of the force $\hat{F}_{2}^{(3)}$, which describes the reciprocal attraction between $D_{1}$ and $D_{2}$, is rather strong, and tends to bring the two disclinations one near the other. In fact, since $\hat{F}_{2}^{(3)}$ is negative, it tends to push $D_{2}$ to the left (see Figure \ref{fig_dipolo}). Also $\hat{F}_{2}^{(2)}$, being negative, cooperates to displace $D_{2}$ leftwards. However, since $\hat{F}_{2}^{(1)}$ is positive and greater in magnitude than the sum between $\hat{F}_{2}^{(2)}$ and $\hat{F}_{2}^{(3)}$, the overall behavior of $D_{2}$ at the initial instants of time is a motion towards the right, 
which is consistent with the positive sign of the total force $\hat{F}_{2}$ in $\mathcal{U}^{+}(0)$, as shown in the right panel of Figure \ref{fig_dipolo}. Note also that $D_{1}$ is pulled to the right, since $\hat{F}_{1}$ is positive. As time goes by, $\hat{F}_{2}$ decreases monotonically, passes through zero, becomes negative, reaches a global minimum, and then grows again, thereby becoming positive. In correspondence to its global minimum, which is around $T=0.2$, the magnitude of $\hat{F}_{2}$ attains the maximum attractive strength, and the same occurs for $\hat{F}_{1}$, which reaches the absolute maximum at the same instant of time at which $\hat{F}_{2}$ is minimum and, thus, gives $D_{1}$ the strongest pull. This behavior explains why $D_{2}$ moves towards the left within the time interval bounded by the two consecutive zeroes of $\hat{F}_{2}$. Afterwards, as $\hat{F}_{2}$ becomes positive, and remains such, it pulls $D_{2}$ towards the right, and so does $\hat{F}_{1}$ with $D_{1}$, by means of a force having approximately the same intensity as $\hat{F}_{2}$ (see Figure \ref{fig_dipolo}). This is the reason why the two disclinations move close to each other, and no collision occurs. We emphasize that, within this time window, the attractive force $\hat{F}_{2}^{(1)}$ (marked with triangles in Figure \ref{fig_dipolo}) contributes the most to the motion of $D_{2}$, thereby making the influence of the boundary stronger than the other forces.


\begin{figure}[htbp]
    \centering
    \begin{minipage}{0.32\textwidth}
        \centering
        \includegraphics[width=\textwidth]{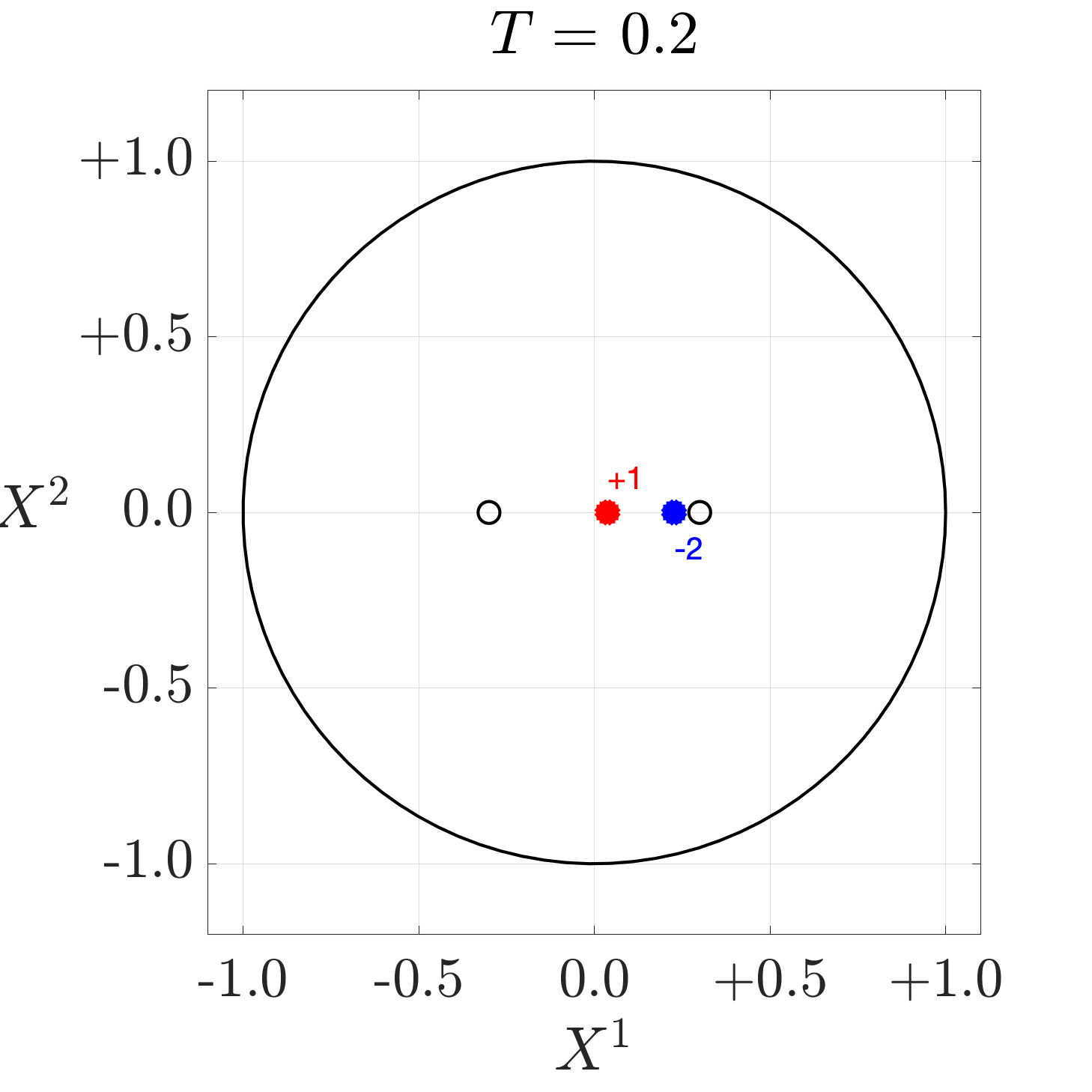}
    \end{minipage}\hfill
    \begin{minipage}{0.32\textwidth}
        \centering
        \includegraphics[width=\textwidth]{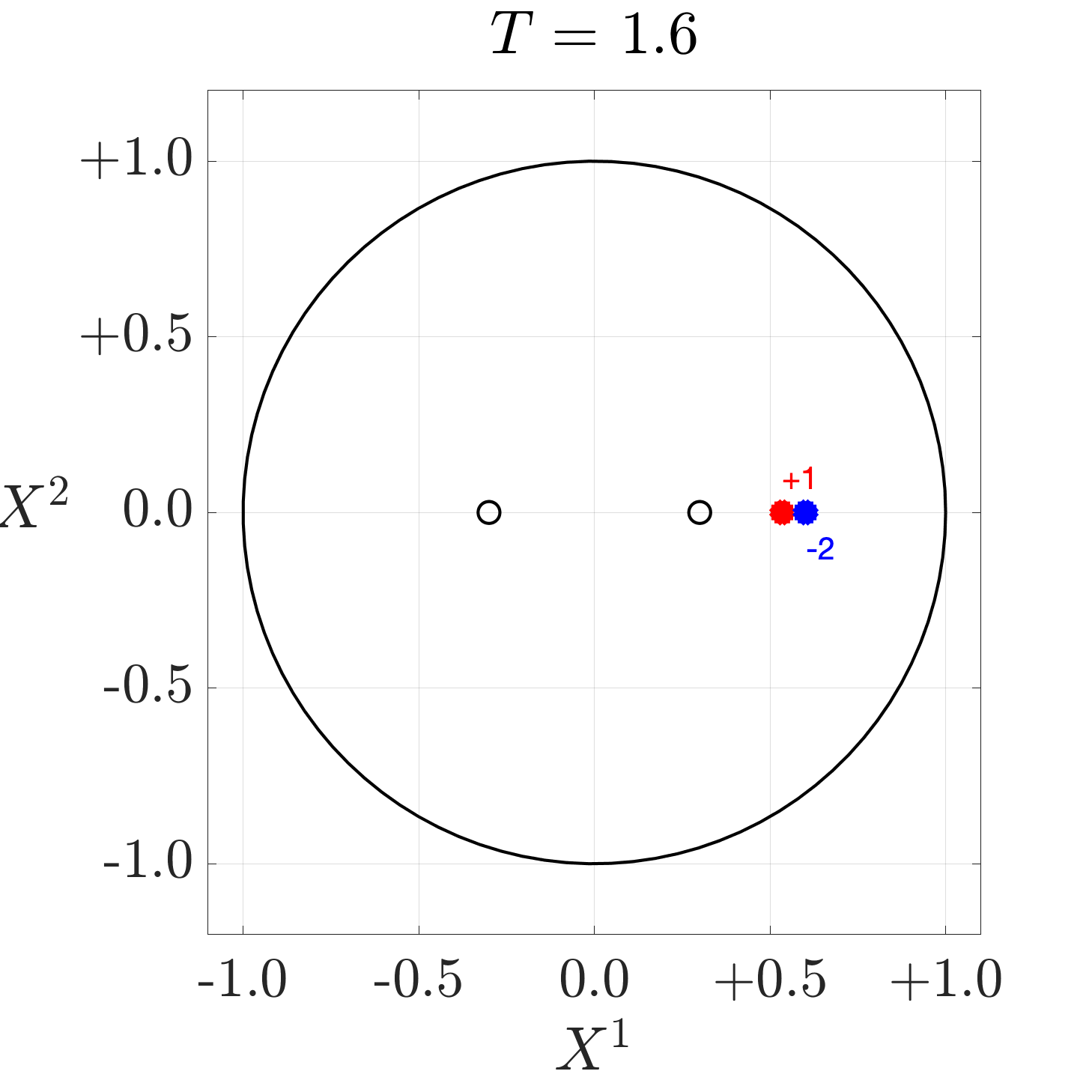}
    \end{minipage}\hfill
    \begin{minipage}{0.32\textwidth}
        \centering
        \includegraphics[width=\textwidth]{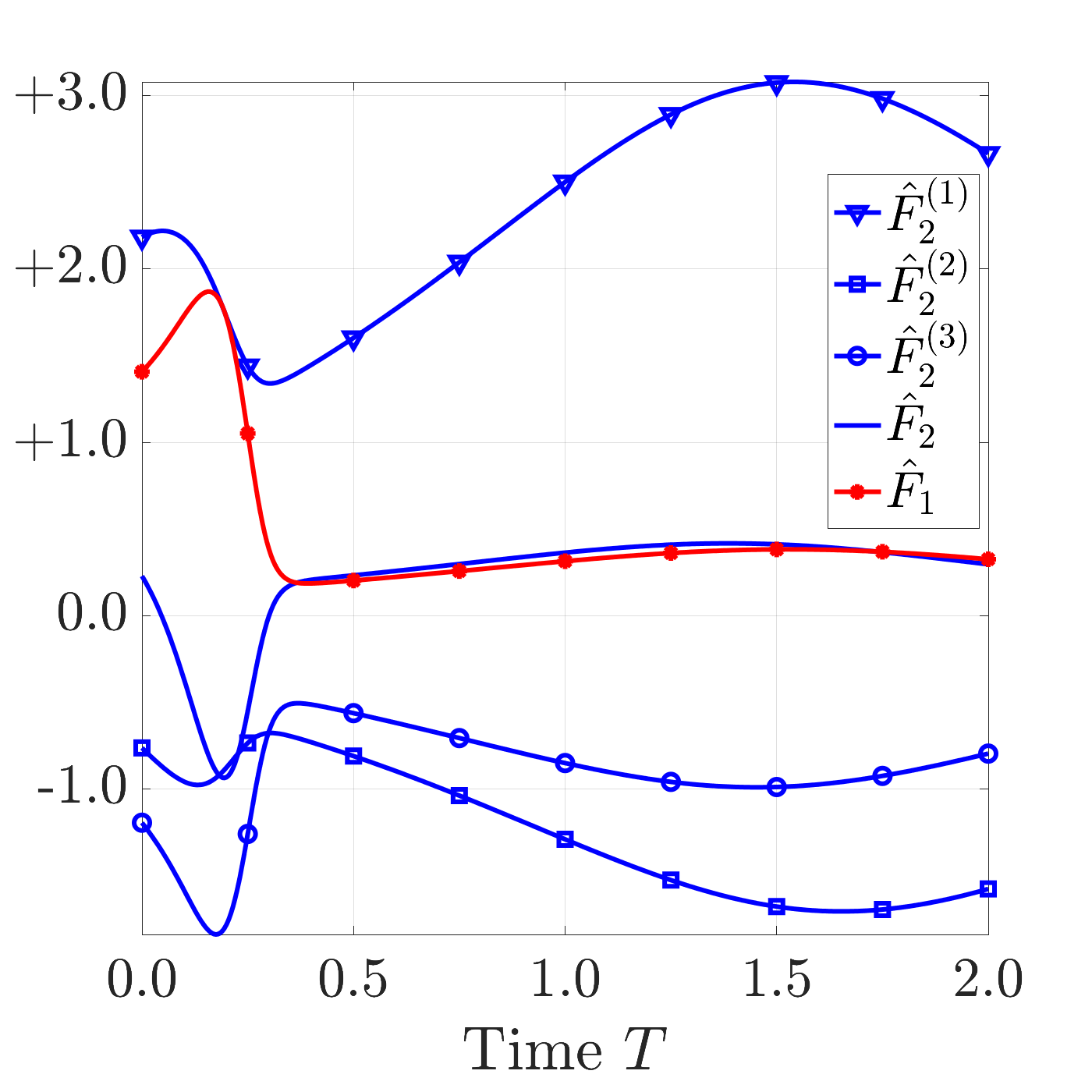}
    \end{minipage}
    
    \caption{(Left and center) Configurations at $T = 0.2$ and $T = 1.6$ for benchmark \#1. The hollow circles refer to $T=0$. (Right) Time trend of  $\hat{F}_{1}$, $\hat{F}_{2}$ and its summands.}
    \label{fig_dipolo}
\end{figure}

\paragraph{\#2. Superposition breaking due to the unbalance of the Frank angles (Figure \ref{fig_superposition_equivalent})} We solve again \eqref{eq_CauchyProblem_Adimensionalized_N2} in the case in which two disclinations with different Frank angles are placed in the same point at the initial time, and are not constrained to evolve with persisting superposition (see Remark \ref{rem_LimitCase_2}). The simulations confirm the predictions of Remark \ref{rem_LimitCase_2}, since they naturally tend to separate, so that the inequality $\Xi_{1}(T)\neq \Xi_{2}(T)$ for $T>0$ holds true. This is because the difference in the Frank angles produces different interactions with $\partial B_{1}(0)$, as visualized in Figure \ref{fig_superposition_equivalent}. 

\begin{figure}[htbp]
    \centering
    \begin{minipage}{0.32\textwidth}
        \centering
        \includegraphics[width=\textwidth]{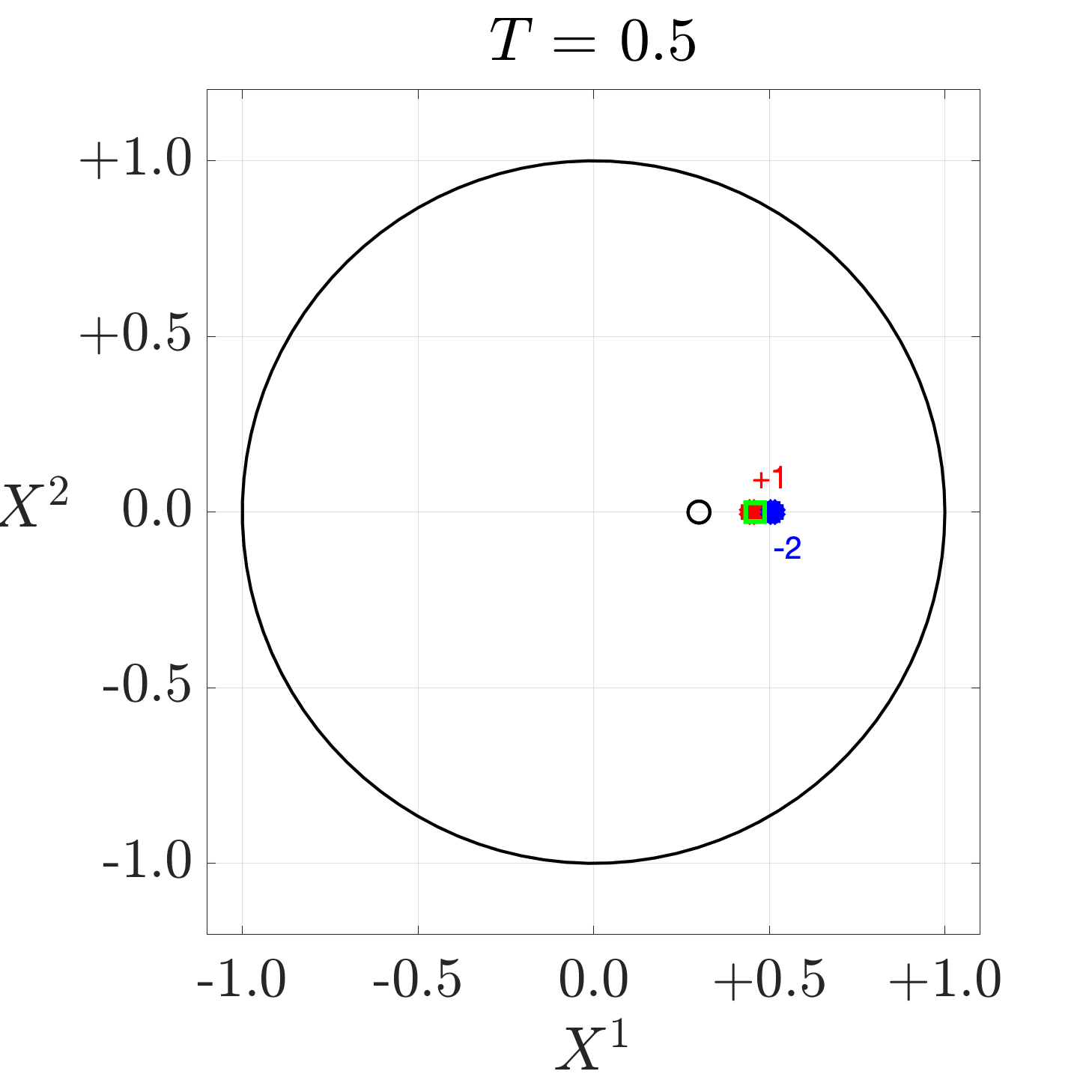}
    \end{minipage}\hfill
    \begin{minipage}{0.32\textwidth}
        \centering
        \includegraphics[width=\textwidth]{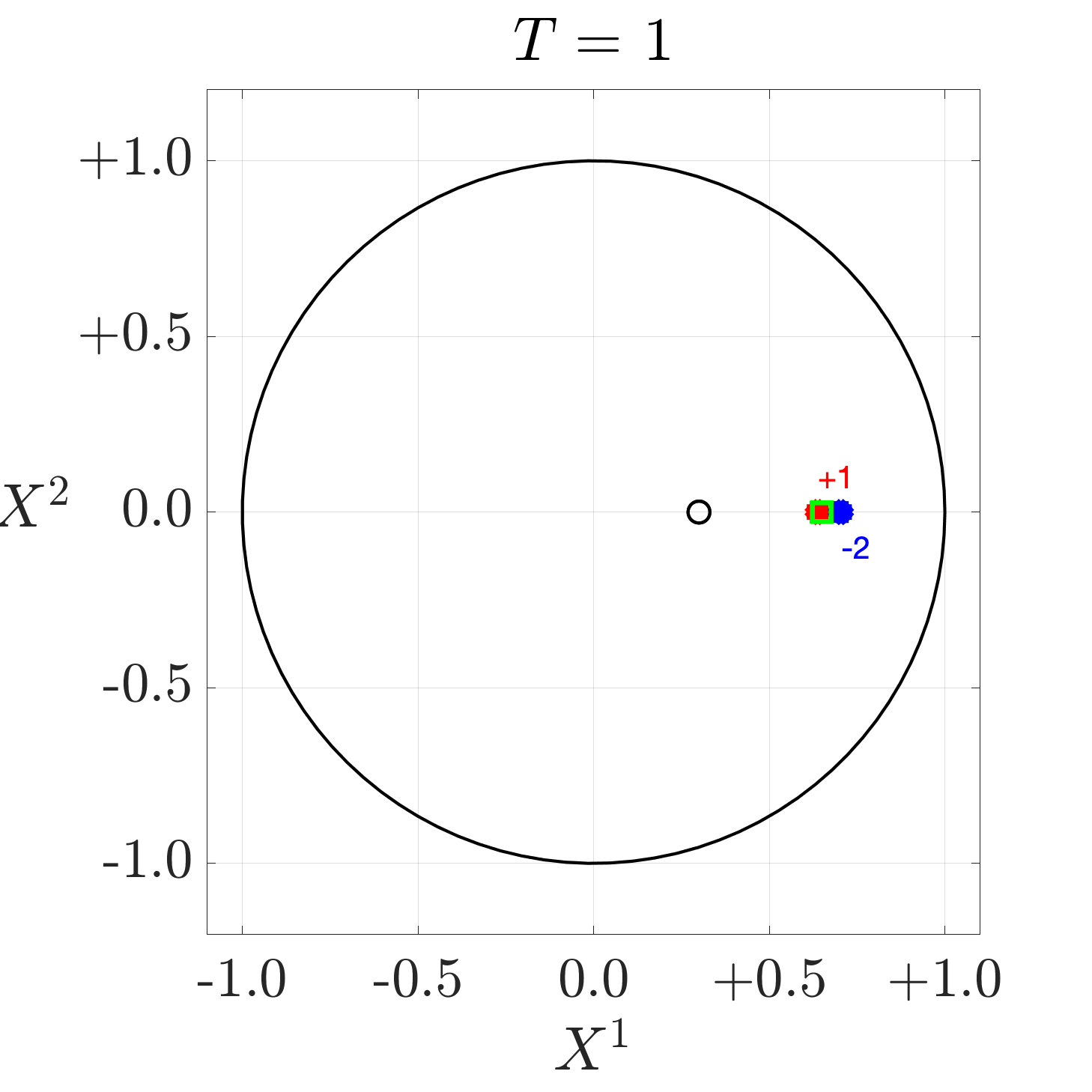}
    \end{minipage}\hfill
    \begin{minipage}{0.32\textwidth}
        \centering
        \includegraphics[width=\textwidth]{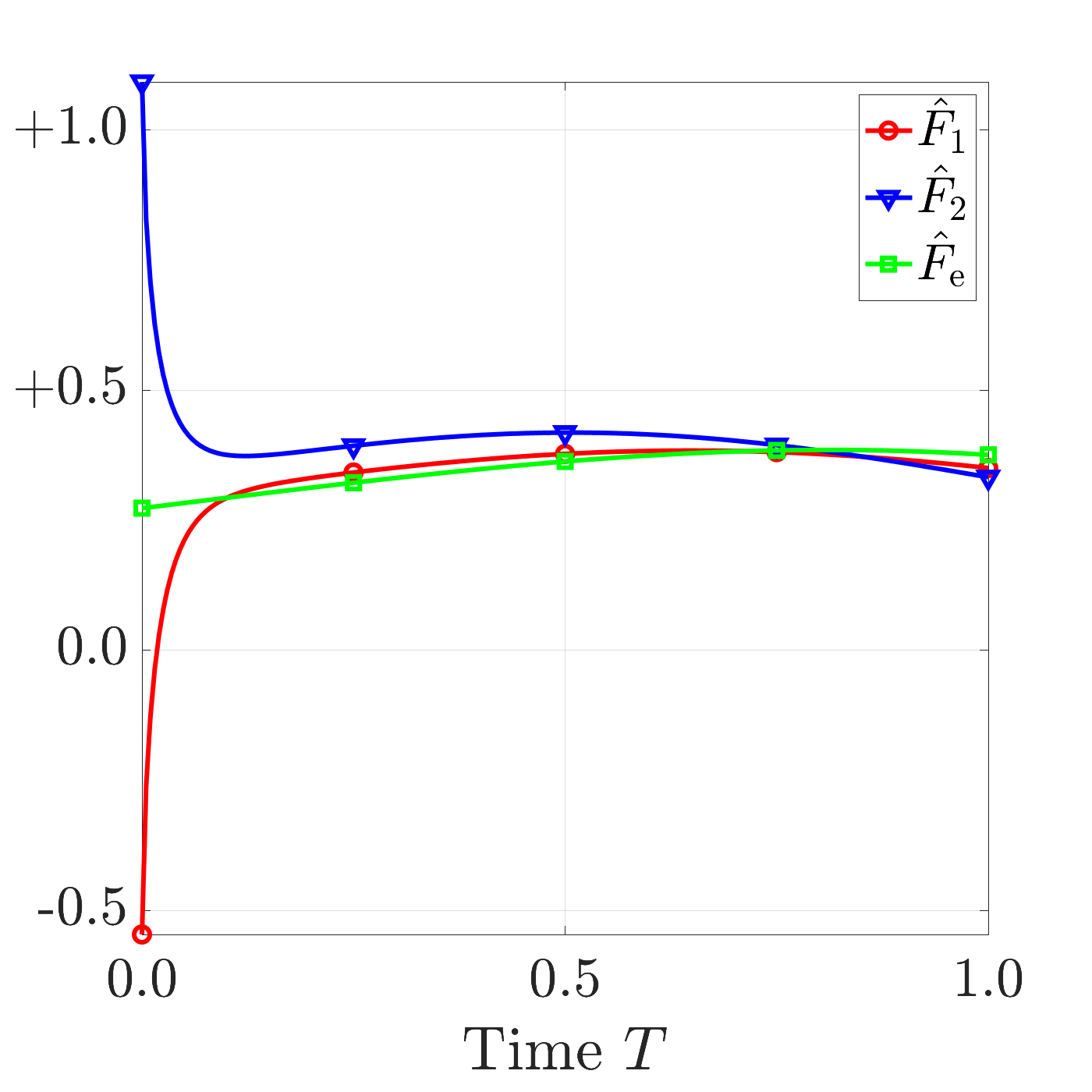}
    \end{minipage}
    
    \caption{(Left and center) Configurations at $T = 0.5$ and $T=1$ for benchmark~\#2. The hollow circles refer to $T=0$; the hollow square refers to the associated ``effective'' disclination (\eqref{eq_effective_equations} and Remark \ref{rem_LimitCase_2}). (Right) Time trend of $\hat{F}_{1}$, $\hat{F}_{2}$ and $\hat{F}_{\mathrm{e}}$.}
    \label{fig_superposition_equivalent}
\end{figure}

\paragraph{\#3. Superposition breaking due to a third disclination (Figure \ref{fig_superposition})} We consider the case of three disclinations, hereafter denoted by $D_{1}$, $D_{2}$, and $D_{3}$, so that $s_{1} = +1$, $s_{2} = -1$, and $s_{3} = +1$. Moreover, we assume $D_{1}$ and $D_{2}$ to be initially superposed, so that their initial common position is, for example,  $\Xi_{10} = \Xi_{20}= 0.3$, and we further set $\Xi_{30} = -0.3$. To investigate this situation, we solve \eqref{eq_CauchyProblem_Adimensionalized} for $N=3$. As anticipated in Section \ref{sec_unfixing}, the presence of $D_{3}$ breaks the initial superposition of $D_{1}$ and $D_{2}$ by attracting $D_{2}$ and repelling $D_{1}$ due to the signs of the considered Frank angles, thereby \emph{splitting} $D_{1}$ and $D_{2}$. However, as $D_{3}$ approaches the boundary, these interactions tend to vanish, thereby making $D_{1}$ and $D_{2}$ approach each other.

\begin{figure}[htbp]
    \centering
    \begin{minipage}{0.32\textwidth}
        \centering
        \includegraphics[width=\textwidth]{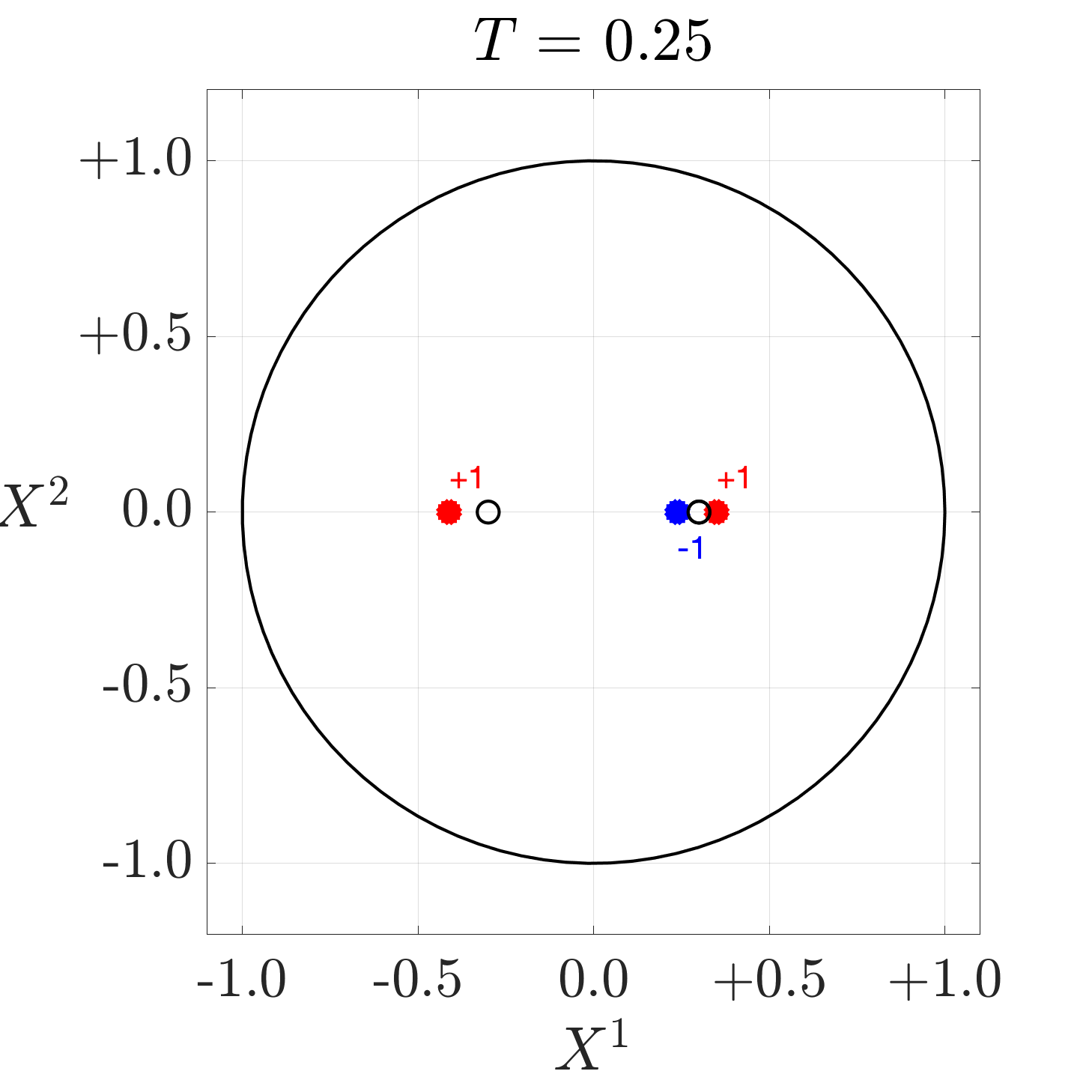}
    \end{minipage}\hfill
    \begin{minipage}{0.32\textwidth}
        \centering
        \includegraphics[width=\textwidth]{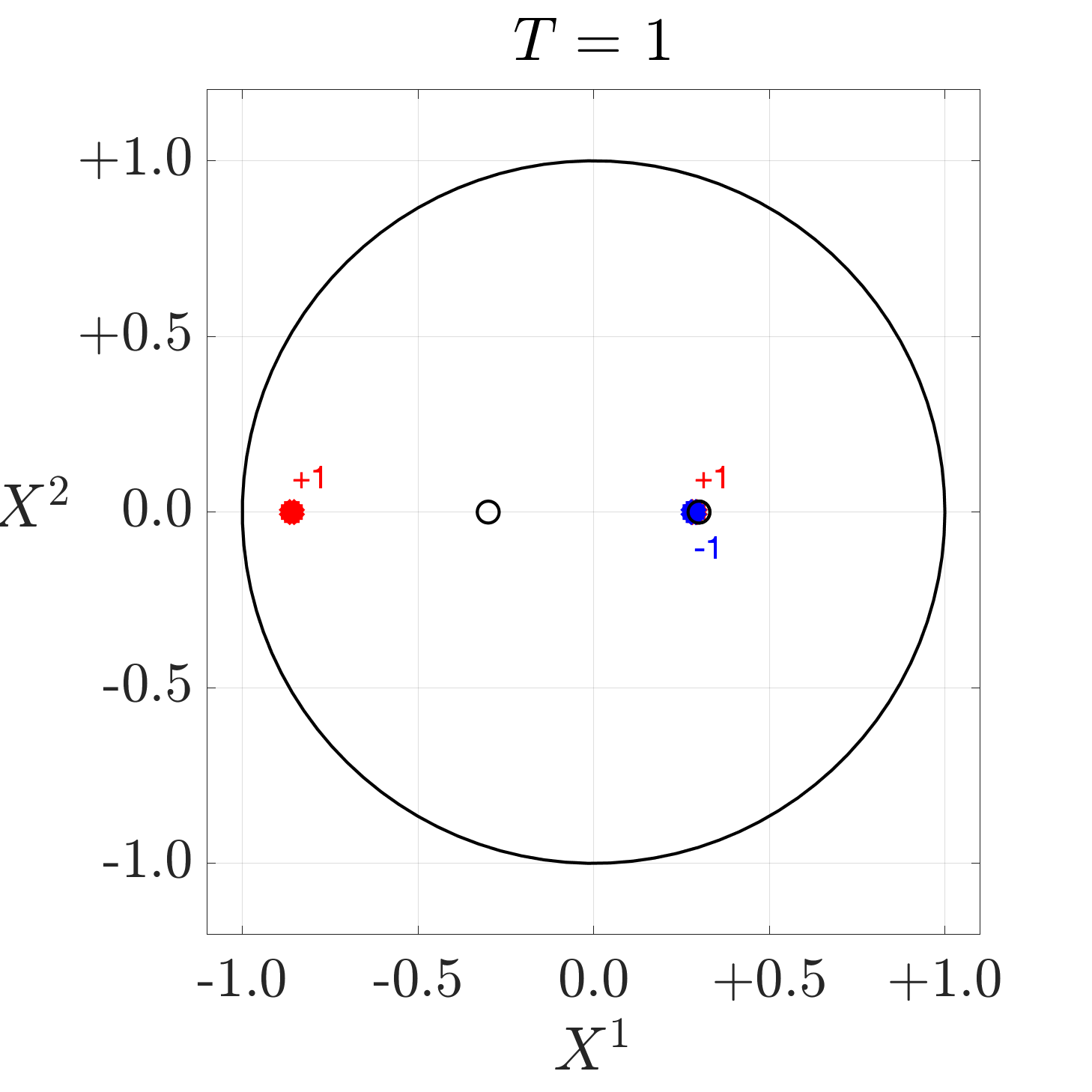}
    \end{minipage}\hfill
    \begin{minipage}{0.32\textwidth}
        \centering
        \includegraphics[width=\textwidth]{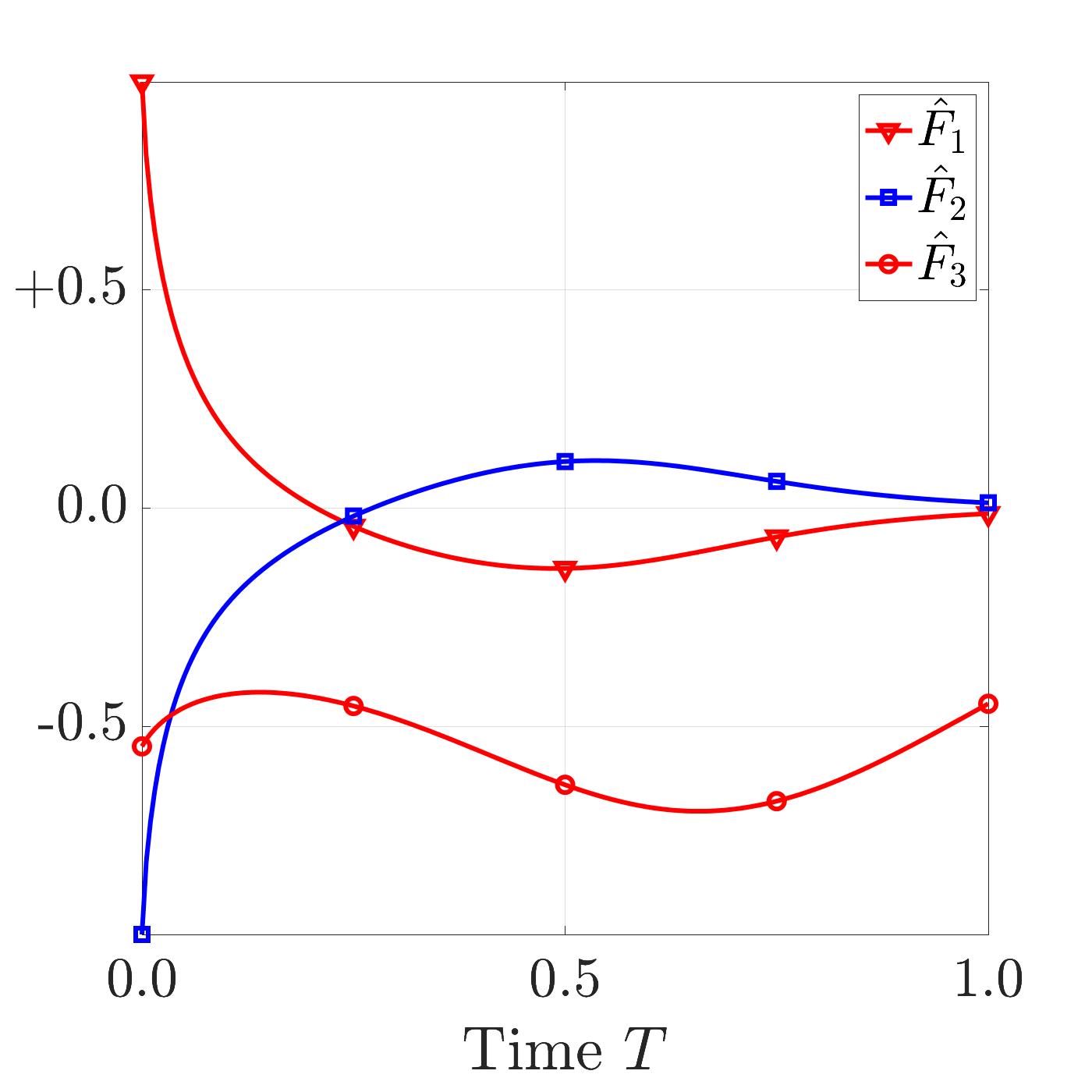}
    \end{minipage}
    
    \caption{(Left and center) Configurations at $T=0.25$ and $T=1$ for benchmark \#3. The hollow circles refer to $T=0$. (Right) Time trend of the forces $\hat{F}_{1}$,  $\hat{F}_{2}$ and $\hat{F}_{3}$.}
    \label{fig_superposition}
\end{figure}

\paragraph{\#4. Three disclinations at the vertices of an equilateral triangle (Figure \ref{fig_tripolo})} We study the case of the three disclinations $D_{1}$, $D_{2}$, and $D_{3}$, initially placed at the vertices of an equilateral triangle centered at the origin of $B_{1}(0)$, and with sides of non-dimensional length $L_{0} = 0.48$, as shown in Figure \ref{fig_tripolo}. We solve \eqref{eq_CauchyProblem_Adimensionalized} by assuming the Frank angles of the three disclinations to be $s_{1} = s_{3} = +1$ and $s_{2} = -1$, so that the total Frank angle of the system is $s_{1} + s_{2} + s_{3} = +1$. We refer to this scenario as \emph{triplet} of disclinations. In a small right neighborhood of $T=0$, the disclinations $D_{1}$ and $D_{3}$ repel each other (they have Frank angles of the same sign). In doing this, they do not move along the segment $\overline{\Xi_{10}\Xi_{30}}$ connecting their initial positions. Rather, they tend to move towards $D_{2}$. In the meanwhile, $D_{2}$ moves towards the middle point of $\overline{\Xi_{10}\Xi_{30}}$, since the attractive forces acting on $D_{2}$ due to $D_{1}$ and $D_{3}$ predominate over the influence of the boundary on $D_{2}$. However, as $D_{2}$ gets closer to the other two disclinations, $D_{1}$ and $D_{3}$ tend to align again with $D_{2}$. This occurs around time $T=5$, as shown in the second panel of Figure \ref{fig_tripolo}. Afterwards, the three disclinations move towards the boundary by placing themselves at the vertices of a triangle that is maintained by the balance among the repulsive force characterizing the pair $(D_{1}, D_{3})$ and the attractive forces characterizing the pairs $(D_{1}, D_{2})$ and $(D_{3}, D_{2})$. We find this result interesting because the disclinations $D_{1}$ and $D_{3}$ tend to migrate towards the boundary very close to each other in spite of the repulsive forces between them. This is made possible by the presence of $D_{2}$, which, having opposite Frank angle, mitigates their mutual repulsion. To us, this behavior is reminiscent of Cooper pairs in superconductors \cite{Fellsager1998}, although the physics of our problem is very different.

\begin{figure}[htbp]
    \centering
    \begin{minipage}{0.32\textwidth}
        \centering
        \includegraphics[width=\textwidth]{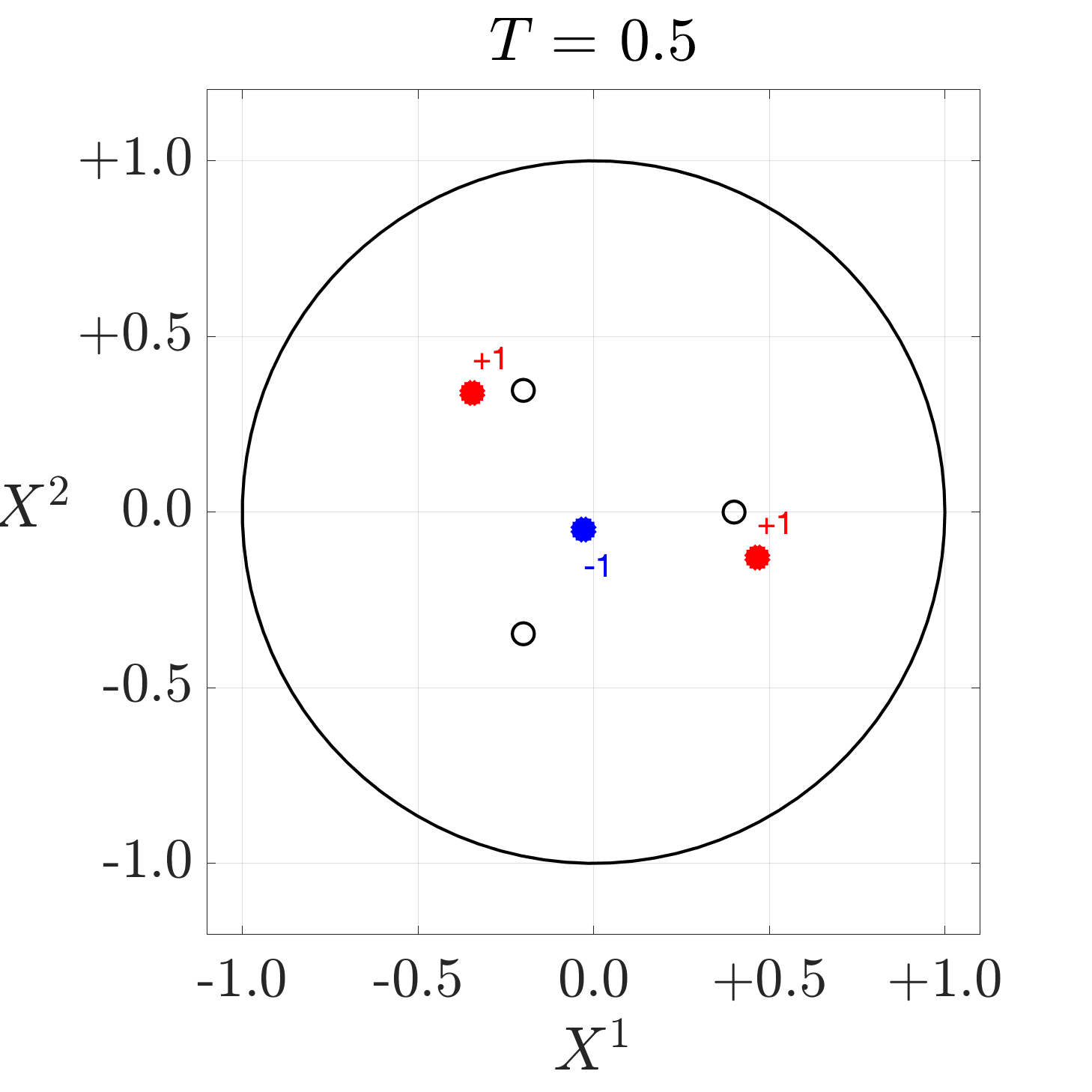}
    \end{minipage}\hfill
    \begin{minipage}{0.32\textwidth}
        \centering
        \includegraphics[width=\textwidth]{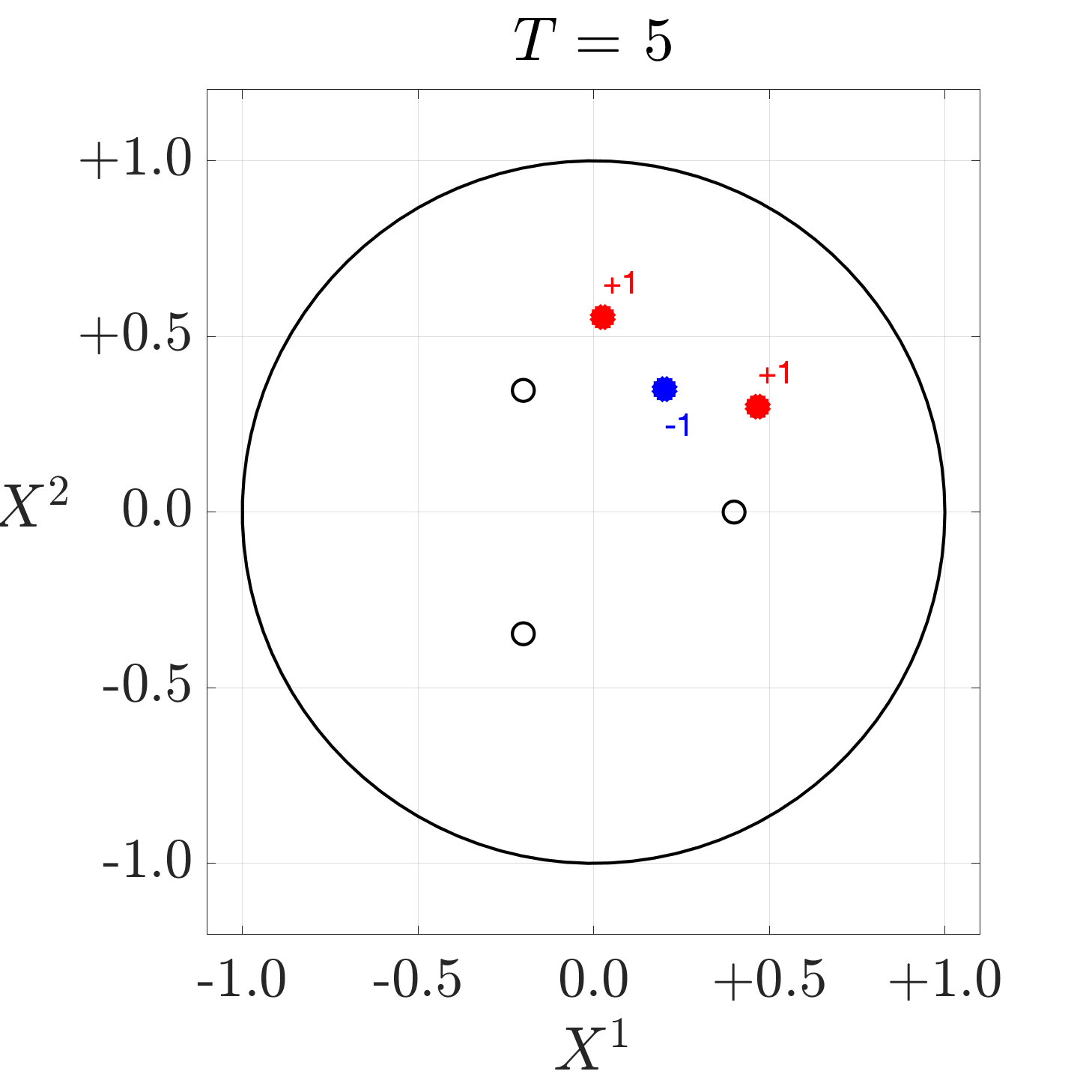}
    \end{minipage}\hfill
    \begin{minipage}{0.32\textwidth}
        \centering
        \includegraphics[width=\textwidth]{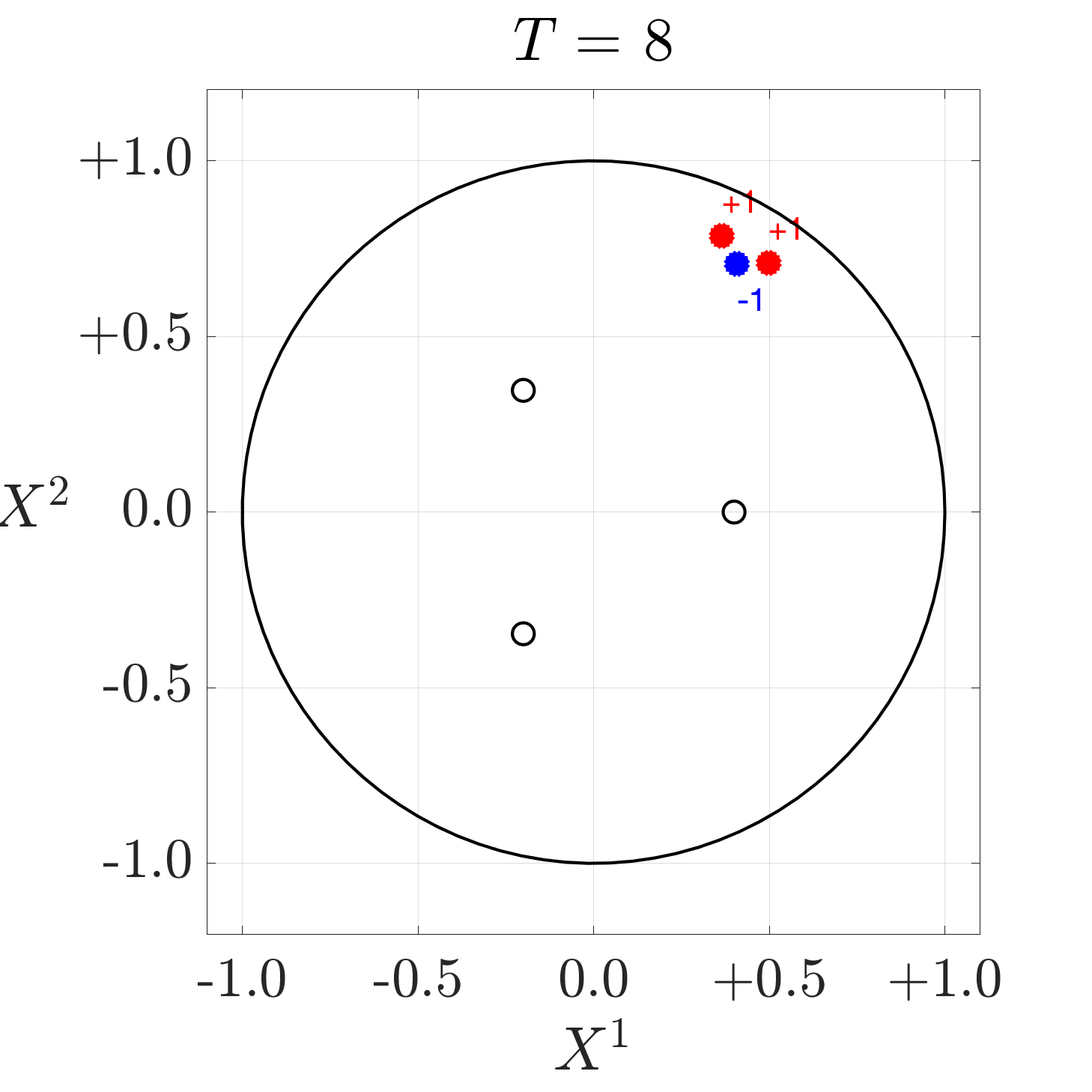}
    \end{minipage}
    \caption{Configurations at $T = 0.5$, $T = 5$ and $T = 8$ for benchmark \#4. The hollow circles refer to $T=0$.}
    \label{fig_tripolo}
\end{figure}

\paragraph{\#5. Seven disclinations with non-zero total Frank angle (Figure \ref{fig_7polo})} As a final scenario for this work, we present the case of $N=7$ disclinations, initially distributed as the vertices of a regular heptagon centered in the origin. For each $k = 1,\ldots, 7$, the Frank angle $s_{k}$ of the $k$th disclination reads $s_{k} = (-1)^{k-1}$, thereby making the total Frank angle of the system to be equal to $s_{1}+\ldots+s_{7} = +1$. Also for this problem, we solve \eqref{eq_CauchyProblem_Adimensionalized}, for $N=7$. In our opinion, this benchmark is remarkable since, by virtue of the \emph{superposition principle}, the system's evolution is characterized by two phenomena that we have individually addressed in this work. In fact, as shown in Figure \ref{fig_7polo}, the first phenomenon that can be observed is that a \emph{triplet} of disclinations is formed in the opposite side to where the two disclinations having Frank angles of the same sign were initially located. In particular, the \emph{triplet} moves towards the boundary as presented in the paragraph above. The second phenomenon depicted in Figure \ref{fig_7polo} is that the remaining four disclinations naturally divide in two pairs that tend to \emph{collide} in the limit for $T\to+\infty$ as presented in Section \ref{sec_TwoDisclinations}.

\begin{figure}[htbp]
    \centering
    \begin{minipage}{0.32\textwidth}
        \centering
        \includegraphics[width=\textwidth]{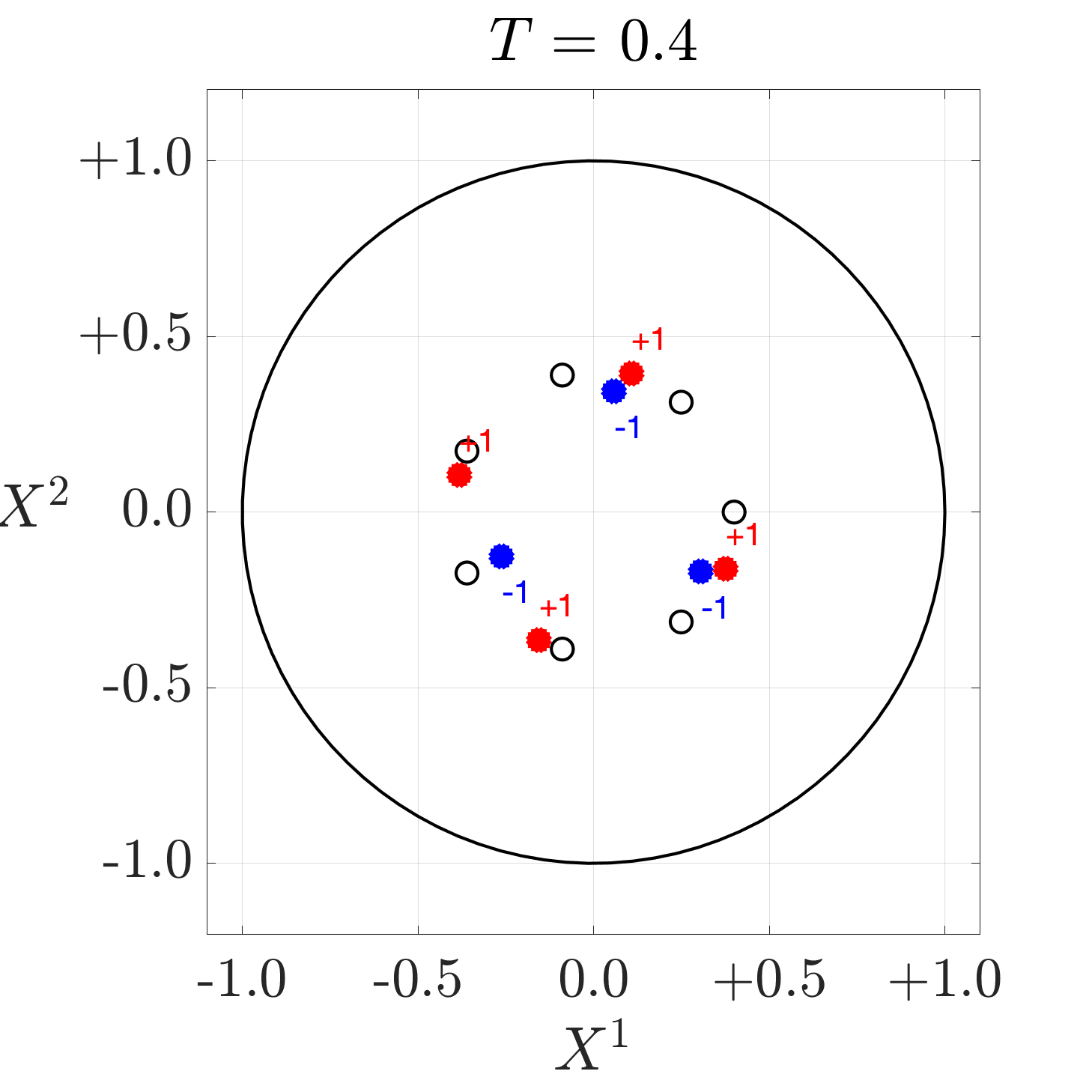}
    \end{minipage}\hfill
    \begin{minipage}{0.32\textwidth}
        \centering
        \includegraphics[width=\textwidth]{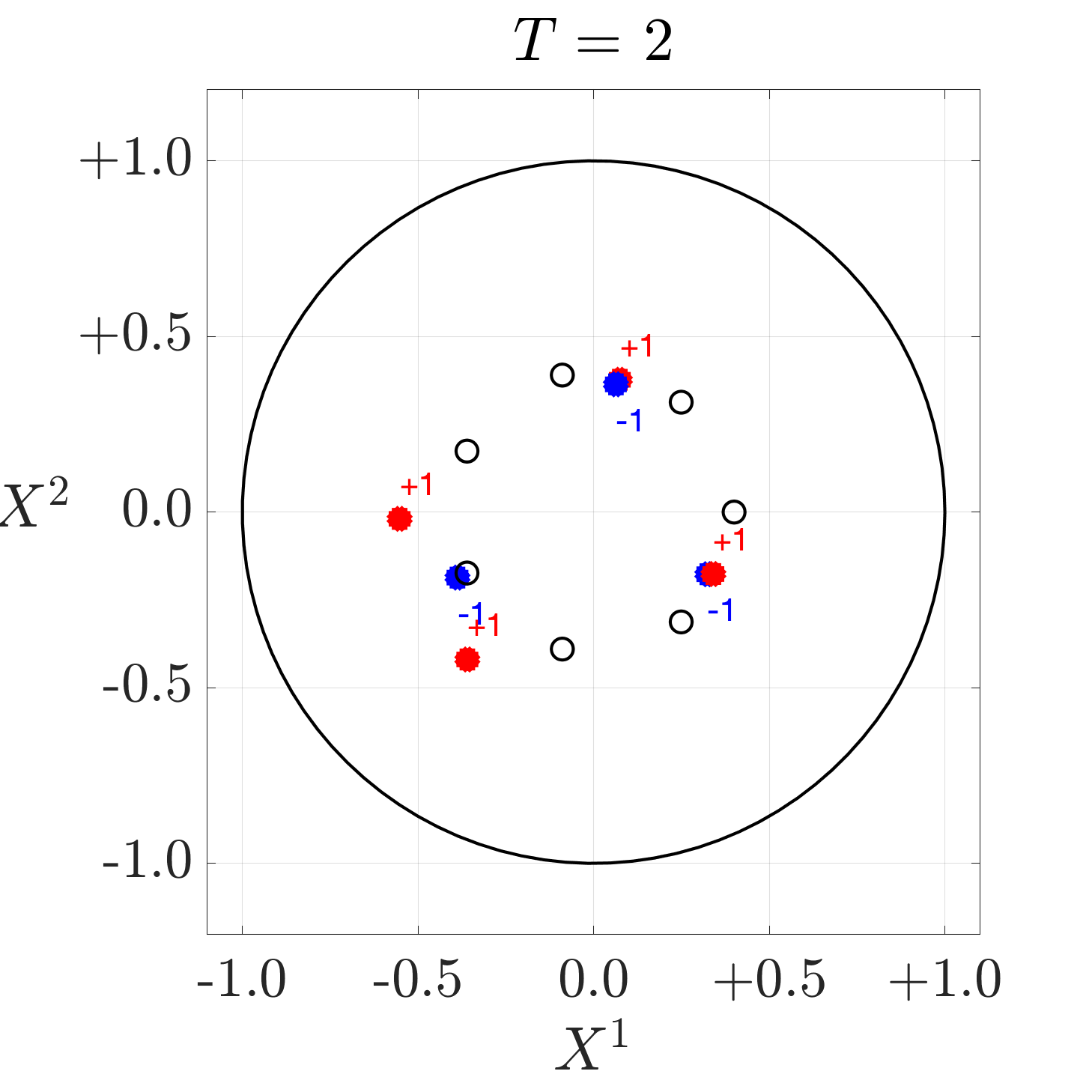}
    \end{minipage}\hfill
    \begin{minipage}{0.32\textwidth}
        \centering
        \includegraphics[width=\textwidth]{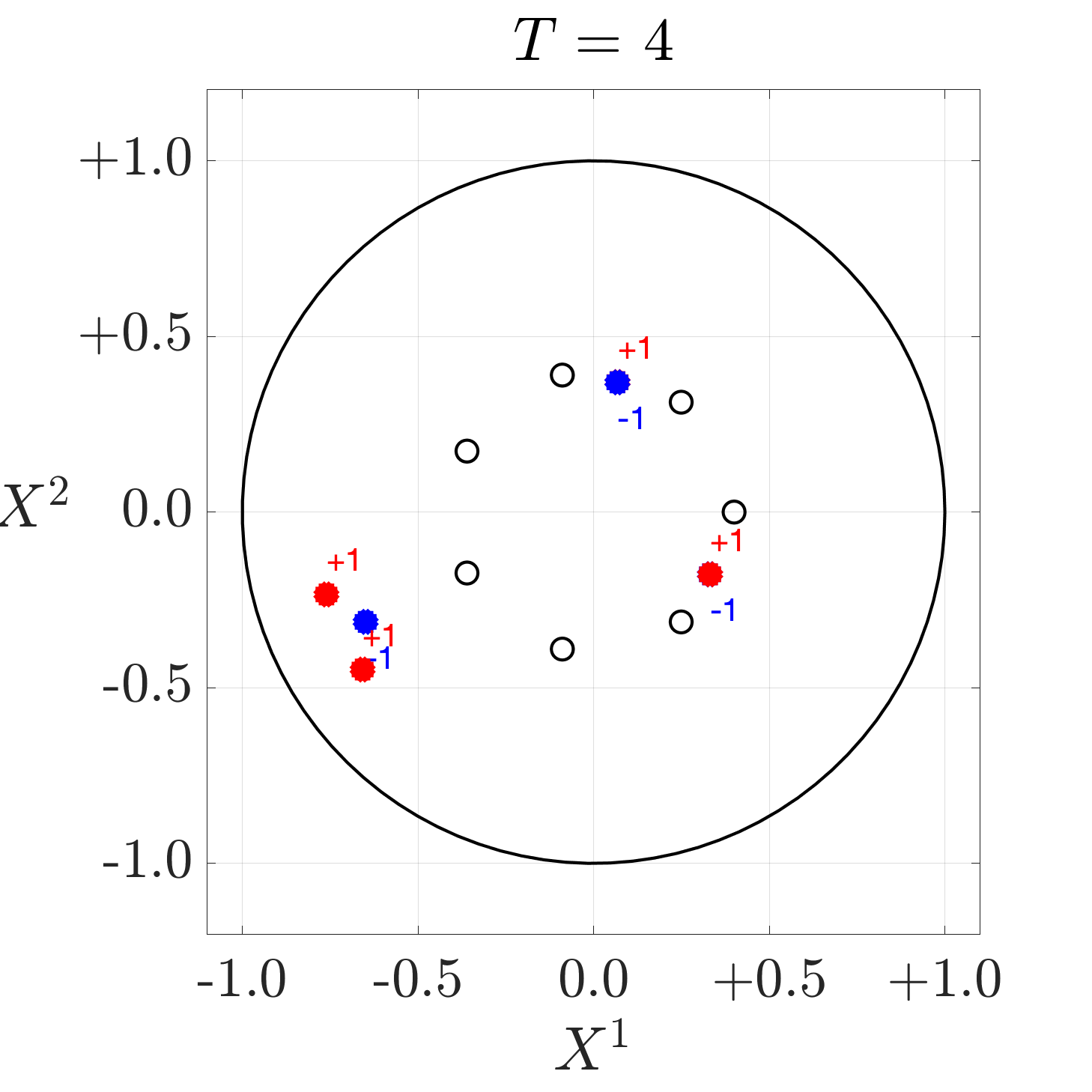}
    \end{minipage}
    \caption{Configurations at $T = 0.4$, $T = 2$ and $T = 4$ for benchmark \#5. The hollow circles refer to $T=0$.}
    \label{fig_7polo}
\end{figure}

\bigskip


\noindent\textbf{Acknowledgments} 
{JSPS Innovative Area grant JP21H00102 and JSPS Grant-in-Aid for Scientific Research (C) JP24K06797 (PC).
PC holds an honorary appointment at La Trobe University and is a member of GNAMPA. 
European Union—Next Generation EU and the Research Projects PRIN2022 PNRR of National Relevance on \emph{Innovative multiscale approaches, possibly based on Fractional Calculus, for the effective constitutive modeling of cell mechanics, engineered tissues, and metamaterials in Biomedicine and related fields} (P2022KHFNB) (AG) and on \emph{Geometric-Analytic Methods for PDEs and Applications} (2022SLTHCE) (MM). This manuscript reflects only the authors’ views and opinions and the Italian Ministry cannot be considered responsible for them.
JASSO scholarship offered within the “Kyushu University Program for Emerging Leaders in Science” (Q‐PELS) and local support received at the Institute of Mathematics for Industry, an international Joint Usage and Research facility located at Kyushu University (AP).
}

\bibliographystyle{siamplain}
\bibliography{Bibliography}
\end{document}